\documentclass{degruyter-proceedings}          %pdflatex

\lastnameone{Combot}
\firstnameone{Thierry}
\nameshortone{T. Combot}
\addressone{Piazza Cavalieri, 56127 Pisa}
\countryone{Italy}
\emailone{thierry.combot@u-bourgogne.fr}

\lastnametwo{}
\firstnametwo{}
\nameshorttwo{}
\addresstwo{}
\countrytwo{}
\emailtwo{}

\newtheorem{defi}{Definition}
\newtheorem{thm}{Theorem}
\newtheorem*{thm*}{Theorem}
\newtheorem{prop}{Proposition}
\newtheorem{cor}{Corollary}
\newtheorem{lem}{Lemma}

\lastnamethree{}
\firstnamethree{}
\nameshortthree{}
\addressthree{}
\countrythree{}
\emailthree{}

\lastnamefour{}
\firstnamefour{}
\nameshortfour{}
\addressfour{}
\countryfour{}
\emailfour{}

\lastnamefive{}
\firstnamefive{}
\nameshortfive{}
\addressfive{}
\countryfive{}
\emailfive{}

\title{Necessary and sufficient conditions for meromorphic integrability near a curve}
\headlinetitle{Meromorphic integrability near a curve}

\abstract{Let us consider a vector field $X$ meromorphic on a neighbourhood of an algebraic curve $\bar{\Gamma}\subset \mathbb{P}^n$ such that $\Gamma$ is a particular solution of $X$. The vector field $X$ is $(l,n-l)$ integrable if it there exists $Y_1,\dots,Y_{l-1},X$ vector fields commuting pairwise, and $F_1,\dots,F_{n-l}$ common first integrals. The Ayoul-Zung Theorem gives necessary conditions in terms of Galois groups for meromorphic integrability of $X$ in a neighbourhood of $\Gamma$. Conversely, if these conditions are satisfied, we prove that if the first normal variational equation $NVE_1$ has a virtually diagonal monodromy group $Mon(NVE_1)$ with non resonance and Diophantine properties, $X$ is meromorphically  integrable on a finite covering of a neighbourhood of $\Gamma$. We then prove the same relaxing the non resonance condition but adding an additional Galoisian condition, which in fine is implied by the previous non resonance hypothesis. Using the same strategy, we then prove a linearisability result near $0$ for a time dependant vector field $X$ with $X(0)=0\;\forall t$.}



\begin{document}

%-------------------------------------------------------------------------%
% To include a table of chapters, write                                   %
%   \tableofcontents
%-------------------------------------------------------------------------%

\section{Integrability}

In this article, we consider an analytic curve $\Gamma$ in the $n$-dimensional complex projective space $\mathbb{P}^n$ such that $\bar{\Gamma}$ is an algebraic curve, and a vector field $X$ meromorphic on $\Omega \subset \mathbb{P}^n$, a neighbourhood of $\bar{\Gamma}$, such that $\Gamma$ is an orbit of $X$. So $\Gamma$ is of the form
$$\Gamma=\{(x_1(t),\dots,x_n(t)),\;\; t\in\mathbb{C}\}$$
where $x(t)$ is a solution of $X$. The curve $\Gamma$ is completed by taking the algebraic closure $\bar{\Gamma}$ in $\mathbb{P}^n$. The points $\bar{\Gamma} \setminus \Gamma$ are limits of the solution $x(t)$ when $t$ goes to infinity or a singularity of $x(t)$, and are equilibrium and singular points of $X$.

We want to study integrability of $X$ in the following sense.

\begin{defi}[See \cite{105}]
We say that $X$ is meromorphically $(l,n-l)$ integrable on $\Omega \subset \mathbb{P}^n$ if there exists meromorphic vector fields $Y_1,\dots,Y_{l-1}$ on $\Omega$ and functions $F_1,\dots,F_{n-l}$ meromorphic on $\Omega$ such that
\begin{itemize}
\item The vector fields $X,Y_1,\dots,Y_{l-1}$ are independent over the meromorphic functions on $\Omega$ and pairwise commute.
\item The functions $F_1,\dots,F_{n-l}$ are functionally independent and are common first integrals of the vector fields $X,Y_1,\dots,Y_{l-1}$.
\end{itemize}
\end{defi}

Let us now define variational equations. The first order variational equation $VE_1$ is given by
$$\dot{Z}=\nabla X(x(t)) Z,\qquad Z\in\mathbb{C}^n.$$
where $x(t)$ is the solution on $\Gamma$ depending on time. This is a linear time depending differential equation. We will note $K$ the differential field generated by $x_1(t),\dots,x_n(t)$, which will serve as the base field for Galois group computation. This differential system always admits a solution in $K$, corresponding to perturbations tangential to the orbit $\Gamma$. The first order variational equation can then be quotiented by this solution, defining a $n-1$ dimensional system, called the normal variational equation $NVE_1$.

If we try to define similarly variational equation of higher order by consider series expansion of $X$ up to order $k$, we end up with non linear equations, which is not satisfactory for Galois group computations. So we define the higher variational equation $VE_k$ of order $k$ as a linear differential equation on the jet space of order $k$ near the point $x(t)$, see \cite{2,81}. An efficient way to build this higher variational equation is to write the series expansion
$$X_j(x_1(t)+Z_1,\dots,x_n(t)+Z_n)= \left(\sum\limits_{i_1,\dots,i_n} f_{j,i_1,\dots,i_n}(t) Z_1^{i_1}\dots Z_n^{i_n}\right)_{j=1\dots n},$$
introduce the variables $Z_{i_1,\dots,i_n}$, write the system
$$\dot{Z}_{m_1,\dots,m_n}= \sum\limits_{j=1}^n m_j Z_j^{m_j-1} \sum\limits_{i_1,\dots,i_n} f_{j,i_1,\dots,i_n}(t) Z_1^{i_1}\dots Z_n^{i_n}$$
and then make the substitution of the monomials in the right part
$$Z_1^{i_1}\dots Z_n^{i_n} \rightarrow Z_{i_1,\dots,i_n} \hbox{ if } i_1+\dots+i_n\leq k $$
$$Z_1^{i_1}\dots Z_n^{i_n} \rightarrow 0 \hbox{ if } i_1+\dots+i_n> k$$
This defines a linear system in $Z_{i_1,\dots,i_n}$ with coefficients in $K$. Its solutions are linear combinations of $Z_1(t)^{i_1}\dots Z_n(t)^{i_n}$ where $Z_1(t),\dots,Z_n(t)$ are solutions of the non linear variational equation. Furthermore, a notion of higher normal variational equations $NVE_k$ will be defined in section $2$, generalizing the first order case $NVE_1$.\\

Now to variational equations $VE_k$ can be associated differential Galois groups $\hbox{Gal}(VE_k)$, defined over the base field $K$. The Ayoul Zung Theorem is the following.

\begin{thm*}[Ayoul-Zung \cite{81}]
If the vector field $X$ is meromorphically integrable on a neighbourhood of $\Gamma$, the Galois groups of the variational equations $VE_k$ on $\Gamma$ are virtually Abelian for all $k\in\mathbb{N}^*$.
\end{thm*}

The proof of this theorem is based on \cite{2}, to which it reduces by doubling the dimension. The purpose of this article is to prove the inverse of the Ayoul-Zung Theorem under some generic conditions, i.e. if all Galoisian conditions are satisfied, then the vector field is meromorphically integrable. Let us first remark that a variational equation can be transformed in a system with algebraic coefficients by changing the time. Let us note $\mathbb{C}(\bar{\Gamma})$ the field of rational functions on $\bar{\Gamma}$. We can now introduce a variable $s$ such that (possibly up to permutation of the coordinates)
$$\bar{\Gamma}=\{(\gamma_1(s),\dots,\gamma_{n-1}(s),s),\quad s\in\mathbb{P}\}$$
where the $\gamma_i$ are (multivalued) algebraic functions. They generate the differential field $\mathbb{C}(\bar{\Gamma})$. The $k$-variational equation using the parametrization by $s$ is of the form
$$(S): \qquad  \frac{\partial}{\partial s} Z=A(s) Z $$
with $A$ a matrix with coefficients in $\mathbb{C}(\bar{\Gamma})$.

\begin{defi}
We will say that a differential system $(S)$ with coefficient in $\mathbb{C}(\bar{\Gamma})$ is Fuchsian if all singularities are regular, i.e. all solutions of $(S)$ have singularities with at most polynomial growth. The monodromy group $Mon(S)$ of the differential system $(S)$ with coefficient in $\mathbb{C}(\bar{\Gamma})$ is the group generated by a fundamental matrix of solutions computed on closed loops of the Riemann surface defined by $\mathbb{C}(\bar{\Gamma})$. The Zariski closure of the monodromy group is the Galois group, i.e. $\overline{Mon(S)}=Gal(S)$, see \cite{64}. The identity component of the Galois group is noted $\hbox{Gal}^0(S)$, and the identity component of the monodromy group is then defined by
$$Mon^0(S)=Mon(S) \cap \hbox{Gal}^0(S).$$
\end{defi}

In this article, we will only consider the case when the normal first order variational equation is Fuchsian and its monodromy group is virtually diagonal, i.e. its identity component is generated up to common conjugacy by diagonal matrices.

\begin{defi}
Let $G$ be a group of diagonal matrices. The group $G$ is
\begin{itemize}
\item $k$-resonant for $j$ if
$$\prod\limits_{i=1}^{n-1} M_{ii}^{k_i}=M_{jj},\; \forall M\in G$$
\item Non-resonant if non $k$-resonant for all $j\in\{1,\dots,n\},\; k\in\mathbb{N}^n, \mid k\mid\geq 2$
\item Diophantine if there exist a basis $(\hbox{diag}(\lambda_{i,1},\dots,\lambda_{i,n}))_{i=1\dots p}$ of $G$ such that
$$\sum\limits_{\nu=1}^\infty 2^{-\nu}\ln\left( \max\limits_{\underset{ \epsilon_{j,k}\neq 0}{j=1\dots n,\; 2\leq \mid k\mid \leq 2^\nu}} \epsilon_{j,k}^{-1}  \right)<\infty \quad \hbox{ with }\quad \epsilon_{j,k}=\max\limits_{i=1\dots p} \left| \prod\limits_{l=1}^n \lambda_{i,l}^{k_l} -\lambda_{i,j}\right|$$
\end{itemize}
\end{defi}

These are the same definitions as simultaneous Brjuno condition in \cite{106} and close to \cite{107}. Remark that the Diophantine property is independent of the choice of the basis of $G$. Indeed, a basis change is equivalent to multiplication of $k$ by an element of $GL_n(\mathbb{Z})$, thus multiplying $\mid k \mid$ by a constant. This multiplies at most the sum by a positive constant, and so does not change its finiteness status. Our main results are the following.

\begin{thm}\label{thmmain1}
Let $X$ be a meromorphic vector field of a neighbourhood of $\bar{\Gamma}$, with $\Gamma$ an algebraic solution of $X$. Assume $NVE_1$ is Fuchsian and $Mon^0(NVE_1)$ is diagonal, non-resonant and Diophantine. The vector field $X$ is meromorphically integrable on a finite covering over a neighbourhood of $\Gamma$ if and only if all variational equations near $\Gamma$ have a virtually Abelian Galois group. Moreover, if integrable, the vector field $X$ is then $(l,n-l)$ integrable with $l=dim(Gal^0(NVE_1))+1$.
\end{thm}

A finite covering over a neighbourhood of $\Gamma$ is introduced because when $Gal(NVE_1)$ is not connected, the vector fields we will build could be finitely multivalued on a neighbourhood of $\Gamma$, with $Gal(NVE_1)/Gal^0(NVE_1)$ inducing a finite action on the valuations. An example is given in section 5.1. Conversely, the Ayoul Zung theorem still applies as lifting the vector field $X$ on a finite covering does not change $Gal^0(VE_k)$, on which the conditions of the Ayoul-Zung Theorem holds.

The non resonance condition is necessary to insure that $\hbox{Gal}^0(NVE_k)$ does not grow, which is pivotal in the proof. We can thus remove this resonance condition, but in the other hand reinforce the conditions on the Galois groups $\hbox{Gal}^0(NVE_k)$.

\begin{thm}\label{thmmain2}
Let $X$ be a meromorphic vector field of a neighbourhood of $\bar{\Gamma}$, with $\Gamma$ an algebraic solution of $X$. Assume $NVE_1$ is Fuchsian, $\hbox{Gal}^0(NVE_k) \simeq \mathbb{C}^{l-1}$ for all $k\in \mathbb{N}^*$ and $Mon^0(NVE_1)$ is Diophantine.
Then the vector field $X$ is $(l,n-l)$ integrable on a finite covering over neighbourhood of $\Gamma$.
\end{thm}

This theorem cannot be transformed into an equivalence because $\hbox{Gal}^0(NVE_k) \simeq \mathbb{C}^{l-1}$ for all $k\in \mathbb{N}^*$ is not a necessary condition for integrability. It is however necessary for a stronger property, linearisability.

\begin{thm}\label{thmmain3}
Let $X$ be a time dependant vector field meromorphic on an algebraic finite covering $\mathcal{C}$ of a neighbourhood of $\{0\in\mathbb{C}^n \}\times\mathbb{P}$. Let us note $\pi:\mathcal{C} \mapsto \mathbb{C}^n$ the projection, $\Gamma=\pi^{-1}(0)\setminus S$ where $S$ are singular points of $X$ on $\pi^{-1}(0)$ and $\bar{\Gamma}=\pi^{-1}(0)$. Assume $X=0$ on $\Gamma$, the $NVE_1$ near $\bar{\Gamma}$ is Fuchsian and $Mon^0(NVE_1)$ is diagonal and Diophantine. The vector field $X$ is holomorphically linearisable on a neighbourhood of $\Gamma$ if and only if $\hbox{Gal}^0(NVE_k) \simeq \mathbb{C}^{l-1}$ for all $k\in \mathbb{N}^*$.
\end{thm}

The linearisability problem corresponds to find a time dependant coordinates change, holomorphic on a neighbourhood of $\Gamma$, such that $X$ becomes its linear part in these new coordinates. The connection between linearisability and integrability is already made in \cite{108} for a neighbourhood of an equilibrium point. The linearisation of Theorem \ref{thmmain3} implies integrability on a finite covering of a neighbourhood of $\Gamma$. However, not all integrable vector fields $X$ on a finite covering of a neighbourhood of $\Gamma$ are linearisable, as given by the following example
$$\dot{q}_1=\frac{\alpha}{s}q_1+\frac{1}{s}q_1^2q_2,\quad \dot{q}_2=-\frac{\alpha}{s}q_2-\frac{1}{s}q_1q_2^2,\quad \dot{s}=1$$
The solutions of this system can be written
\begin{equation}\label{eqsol}
q_1=c_1s^{\alpha+c_1c_2},\;q_2=c_2s^{-\alpha-c_1c_2},\; s=t+c_3
\end{equation}
The system is $(2,1)$ integrable with
$$q_1q_2,\;\; q_1\frac{\partial}{\partial q_1}-q_2\frac{\partial}{\partial q_2},\;\; \left(\frac{\alpha}{s}+\frac{1}{s}q_1q_2\right)q_1\frac{\partial}{\partial q_1}-\left(\frac{\alpha}{s}+\frac{1}{s}q_1q_2\right)q_2\frac{\partial}{\partial q_2}+\frac{\partial}{\partial s}$$
However the system is not linearisable near $0$. Indeed, $2$-dimensional linear diagonal systems have solutions using at most $2$ hyperexponential functions, independent of initial conditions, and here the hyperexponential function $s^{\alpha+c_1c_2}$ depends on $c_1,c_2$. Remark that the contraposition of Theorem \ref{thmmain3} is satisfied: the system is not linearisable, and the Galois group of the $NVE_k$ grows. This growing is obtained when differentiating \eqref{eqsol} with respect to $c_1,c_2$, producing logs terms at order $3$ and higher in $c_1,c_2$.

The plan of the article is the following
\begin{itemize}
\item In section 2, we define several reduction of the vector fields and the higher order normal variational equations $NVE_k$.
\item In section 3, we prove formally the right to left implications of the main Theorems \ref{thmmain1},\ref{thmmain2},\ref{thmmain3}.
\item In section 4, we prove that the formal first integrals and vector fields converge on a finite covering over neighbourhood of $\Gamma$. We then finish the proofs of Theorems \ref{thmmain1},\ref{thmmain2},\ref{thmmain3}.
\item In section 5, we present some generalizations under stronger conditions, in particular the completion at singular points $\bar{\Gamma} \setminus \Gamma$ and computation of the covering of $\Gamma$ on which the vector fields and first integrals are defined.
\end{itemize}

\section{Normal higher variational equations}

\subsection{Good coordinates near $\Gamma$}

Let us consider the following coordinates of the neighbourhood of $\bar{\Gamma}$
$$x=(\gamma_1(s)+q_1,\gamma_2(s)+q_2,\dots,\gamma_{n-1}(s)+q_{n-1},s)$$
whose inverse is
$$(q,s)=(x_1-\gamma_1(x_n),\dots,x_{n-1}-\gamma_{n-1}(x_n),x_n)$$
This coordinates become singular when the $\gamma_i$ are singular. Now in these coordinates, the curve $\bar{\Gamma}$ becomes $q=0$. The vector field has $\Gamma$ for solution, and thus is tangent to $\Gamma$. So the vector field can now be written under the form
\begin{equation}\label{eqreduc0}
\dot{q}_1=\sum\limits_{i=1}^{n-1} q_iX_{1,i}(s,q),\;\dots\;, \dot{q}_{n-1}=\sum\limits_{i=1}^{n-1} q_iX_{n-1,i}(s,q),\; \dot{s}=X_n(s,q)
\end{equation}
where the $X_{i,j}$ are meromorphic in $q,s$ on a neighbourhood of $\{0\}\times \bar{\Gamma}$. Moreover, they are analytic on $\Gamma$ except possibly when the $\gamma_i$ or the initial vector field $X$ become singular. The functions $X_{i,j}$ can therefore be represented by series in $q$ with coefficients meromorphic in $s\in \bar{\Gamma}$. Meromorphic functions on $\bar{\Gamma}$ are in fact algebraic, and form the field $\mathbb{C}(\bar{\Gamma})$.\\

\noindent
\textbf{Example}
$$\dot{x}_1=x_2,\;\;\dot{x}_2=-x_1$$
This vector field admits $x_1^2+x_2^2=1$ as an algebraic solution. In the $q,s$ coordinates, we have
$$\dot{q}=-\frac{qs}{\sqrt{1-s^2}},\;\; \dot{s}=-q-\sqrt{1-s^2}$$
The singular locus of $\gamma_1(s)=\sqrt{1-s^2}$ is $s=\pm 1$ which is now a ramified pole of the vector field.

\subsection{Gauge reduction}

If the first order variational equation has a virtually Abelian Galois group, the solutions can be written using only functions of the form
$$e^{\int w(s) ds} \hbox{ or } \int w(s) ds$$
with $w(s)$ an algebraic function on $\bar{\Gamma}$. Let us define what we will call hyperexponential functions and logarithmic functions.

\begin{defi}
A hyperexponential function $H$ on $\bar{\Gamma}$ is a multivalued function such that
$$\frac{\partial}{\partial s}H(s)=h(s) H(s),\quad s\in \bar{\Gamma}$$
where $h(s)$ is an algebraic function on $\bar{\Gamma}$.
A logarithmic function $L$ on $\bar{\Gamma}$ is a multivalued function such that
$$\frac{\partial}{\partial s} L(s)=h(s),\quad s\in \bar{\Gamma}$$
where $h(s)$ is an algebraic function on $\bar{\Gamma}$.
\end{defi}

Let us now recall that also our reduction added singularities in the expression of the first order variational equation. However these singularities are only due to the change of coordinates, the base field is still $\mathbb{C}(\bar{\Gamma})$, and thus Galois group and monodromy group stays unchanged.

When solving the $NVE_1$ with virtually diagonal Galois group, we first consider the field extension $\mathcal{P}\cap \overline{\mathbb{C}(s)}$ where $\mathcal{P}$ is the Picard Vessiot field of $NVE_1$. It is an algebraic extension of the field of rational functions on $\bar{\Gamma}$ and so defines a Riemann surface $\Sigma$ above $\bar{\Gamma}$. We now have 
$$\mathcal{P}\cap \overline{\mathbb{C}(s)}=\mathbb{C}(\Sigma)$$
where $\mathbb{C}(\Sigma)$ is the field of rational functions on $\Sigma$. Now the Picard Vessiot field can be expressed by
$$\mathcal{P}=\mathbb{C}(\Sigma)\left(H_1(s),\dots,H_{l-1}(s)\right)$$
where the $H_i$ are hyperexponential functions with logarithmic derivative in $\mathbb{C}(\Sigma)$.

Remark that to solve the $VE_1$ when knowing the solutions of the $NVE_1$, we just have to integrate some linear combination of the solutions of the $NVE_1$. If $VE_1$ has a virtually Abelian Galois group, we will have to consider at worst an additional logarithmic function for the Picard Vessiot field of the $VE_1$, but no additional algebraic extension will be necessary.

\begin{prop}[See \cite{109} Proposition 3]\label{propgauge}
Assume $Gal^0(NVE_1)$ is diagonal. There exists a gauge transformation with coefficients in $\mathbb{C}(\Sigma)$ of the $NVE_1$ such that the $NVE_1$ is then of the form
$$Z'=\left(\begin{array}{cccc}
\tilde{X}_{1,1}&0& \dots & 0\\
 & \dots & &  \\
0&\dots &0& \tilde{X}_{n-1,n-1}\\
\end{array}\right)Z.$$
\end{prop}

The gauge transformation is a linear transformation on the $Z$'s, but can also be applied to the (non linear) system \eqref{eqreduc}.

\begin{defi}
The gauge transformation of Proposition \ref{propgauge} applied to the variables $q_1,$ $\dots,q_{n-1}$ of equation \ref{eqreduc0} defines a vector field meromorphic in $q_1,\dots,q_{n-1}$ with coefficients in $\mathbb{C}(\Sigma)$ we called gauge reduced.
\end{defi}

After gauge reduction, the system is of the form
\begin{equation}\begin{split}\label{eqsys}
\dot{q}_1 &=\tilde{X}_{1,1}(s)q_1+\sum\limits_{\mid i\mid\geq 2}^\infty f_{1,i}(s) q^i\\
& \qquad  \dots\\
\dot{q}_{n-1} &=\tilde{X}_{n-1,n-1}(s)q_{n-1}+\sum\limits_{\mid i\mid\geq 2}^\infty f_{n-1,i}(s) q^i\\
\dot{s} &= \tilde{X}_n(s,q)
\end{split}\end{equation}
with $f_{j,i}\in\mathbb{C}(\Sigma)$. The poles of the $f_{j,i},\tilde{X}_{i,i},\tilde{X}_n$ in $\Sigma$ are such that their projection on $\bar{\Gamma}$ is either such that $\gamma$ is singular (due to the coordinate system near $\Gamma$) or on $\bar{\Gamma}\setminus \Gamma$ (proper singularities of the initial system). In any case, they belong to a finite set $\mathcal{S}\subset \Sigma$, not depending on $j,i$.

Remark that gauge reduction is simply a variable change, and thus conserves the commuting vector fields and first integrals on a neighbourhood of $\Gamma$ if they exist. However, due to the introduction of the algebraic Riemann surface $\Sigma$, these vector fields and first integrals will be a priori defined on a finite covering of a neighbourhood of $\Gamma$. Thus $(l,n-l)$ integrability on a finite covering of a neighbourhood of $\Gamma$ is conserved by gauge reduction.

\subsection{Time reduction}

We want in equation \eqref{eqreduc0} a normalization of $X_n(s,q)$ to $1$, but this corresponds to a change of time, and commuting vector fields are not invariant by time change. To keep track of the time change, we introduce the time variable $t$, and the equation $\dot{t}=1$. The vector fields are now in dimension $n+1$, and we can make the time change, giving the system
$$\dot{t}=\frac{1}{X_n(s,q)},\;\dot{s}=1,$$
$$\dot{q}_1=\sum\limits_{i=1}^{n-1} q_i\frac{X_{1,i}(s,q)}{X_n(s,q)},\;\dots\;, \dot{q}_{n-1}=\sum\limits_{i=1}^{n-1} q_i\frac{X_{n-1,i}(s,q)}{X_n(s,q)}$$
The last step is to consider $s$ as the new time, removing the equation $\dot{s}=1$
\begin{equation}\label{eqreduc}
\dot{q}_1=\sum\limits_{i=1}^{n-1} q_i\frac{X_{1,i}(s,q)}{X_n(s,q)},\;\dots\;, \dot{q}_{n-1}=\sum\limits_{i=1}^{n-1} q_i\frac{X_{n-1,i}(s,q)}{X_n(s,q)},\;\dot{t}=\frac{1}{X_n(s,q)}
\end{equation}
The vector field obtained is time dependant, and the new time $s$ lives on $\Sigma$.

\begin{defi}
The equation \eqref{eqreduc} defines a vector field we call the time reduction of the vector field $X$.
\end{defi}

The time reduction and gauge reduction can be made independently, and if done so we call the resulting vector field time and gauge reduced. Each of these reductions have drawbacks
\begin{itemize}
\item The gauge reduction requires additional ramifications as the resulting vector field has a series expansion in $q$ with coefficients in $\mathbb{C}(\Sigma)$ instead of $\mathbb{C}(\bar{\Gamma})$
\item The time reduction does not conserve commuting vector fields.
\end{itemize}

Let us remark that time reduction can produce some additional singularities. The function $X_n(s,q)$ corresponds to the tangential part of $X$ along $\Gamma$, and thus vanishes when $X$ vanishes. These are the true singular points on the curve $\bar{\Gamma}$ (and so not in $\Gamma$). However, when $\gamma$ is singular, this also can produce a singularity. Still this singularity is artificial, as it is due only to the coordinates system, and thus the monodromy of the variational equations will be trivial around it. Let us now define the higher normal variational equations

\begin{defi}\label{definve}
The higher normal variational equation $NVE_k$ is the variational equation of \eqref{eqreduc} without the equation in $t$.
\end{defi}

The key point to allow this definition is that $t$ does not appear in the equations in $q_i$, and so restricting the system to the $q$'s has sense. We now need to check that the classical definition of $NVE_1$ coincide with this definition for $k=1$.

\begin{lem}
The above definition of $NVE_k$ when $k=1$ is the $VE_1$ quotiented by the solution corresponding to a tangential perturbation to $\Gamma$ with a change of independent variable.
\end{lem}

\begin{proof}
We begin from equation \eqref{eqreduc0}. The curve $\bar{\Gamma}$ is the straight line $q_1=q_2=\dots=q_{n-1}=0$, so the tangential perturbation is in the direction $s$ only. The first order variational equation is
$$\dot{Z}=\left(\begin{array}{cccc}
\partial_s X_n & \partial_{q_1} X_n& \dots &\partial_{q_{n-1}} X_n\\
0&X_{1,1}&\dots & X_{1,n-1}\\
0& & \dots & \\
0&X_{n-1,1}&\dots & X_{n-1,n-1}\\
\end{array}\right)Z$$
The matrix is already in block triangular form, and the first component of $Z$ correspond to the perturbation in $s$, so tangential to $\Gamma$. So the $VE_1$ quotiented by the solution corresponding to a tangential perturbation to $\Gamma$ is given by the last $n-1$ coordinates, which defines a $(n-1)\times (n-1)$ matrix whose entries are $X_{i,j}(s,0)$.

Let us now compute the $NVE_1$ according to Definition \ref{definve}. It is given by
$$Z'=\frac{1}{X_n}\left(\begin{array}{ccc}
X_{1,1}&\dots & X_{1,n-1}\\
 & \dots & \\
X_{n-1,1}&\dots & X_{n-1,n-1}\\
\end{array}\right)Z$$
with $'$ the differentiation in $s$. The function $X_n(s,0)$ appearing in front of the matrix can be removed using $dt=X_n(s,0) ds$, and so these matrices are the same after a change of independent variable.
\end{proof}

\section{Formal results}

\subsection{Formal flow}

From now on, the $NVE_1$ is assumed to have a virtually diagonal monodromy group and to be Fuchsian (and so the Galois group is also virtually diagonal). When gauge reduced, we note $H_1,\dots,H_{n-1}$ the hyperexponential basis of solutions of the $NVE_1$ with logarithmic derivatives $\tilde{X}_{i,i}$.

\begin{prop}\label{propformal1}
Let us consider $X$ a time and gauge reduced vector field and assume that $Gal^0(VE_k)$ is Abelian $\forall k\in\mathbb{N}^*$. We assume moreover \textbf{at least one} of the following hypotheses
\begin{itemize}
\item $Mon^0(NVE_1)$ is non-resonant.
\item $\hbox{Gal}^0(NVE_k) \simeq \mathbb{C}^{l-1} \;\forall k\in\mathbb{N}^*$
\end{itemize}
Let us consider $s_0$ a regular point on $\Sigma$ with $X(s_0,0)$ non singular. Then there exists formal series
$$\varphi_j(s,c_1,\dots,c_{n-1})=\sum\limits_{i\in\mathbb{N}^{n-1}}  a_{j,i_1,\dots,i_{n-1}}(s) (c_1H_1(s))^{i_1}\dots (c_{n-1}H_{n-1}(s))^{i_{n-1}}$$
for $j=1\dots n-1$ and
$$\varphi_n(s,c_1,\dots,c_{n-1})=\sum\limits_{i\in \hbox{Res}}  c_1^{i_1}\dots c_{n-1}^{i_{n-1}}L_i(s)+$$
$$\sum\limits_{i\in\mathbb{N}^{n-1}}  a_{n,i_1,\dots,i_{n-1}}(s) (c_1H_1(s))^{i_1}\dots (c_{n-1}H_{n-1}(s))^{i_{n-1}}$$
such that $(q_1(s),\dots,q_{n-1}(s),t(s))=\varphi(s,c_1,\dots,c_{n-1})$ is a formal solution of $X$ for any $c$,
$$a_{j,0,\dots,0,1,0,\dots,0}=0 \hbox{ if the }1\hbox{ is not in position }j ,\; \forall j=1\dots n-1$$
$$a_{j,0,\dots,0,1,0,\dots,0}=1 \hbox{ if the }1\hbox{ is in position }j,\; \forall j=1\dots n-1,$$
$$a_{j,m_1,\dots,m_{n-1}}(s_0)=0 \hbox{ if } \hbox{Mon}^0(NVE_1) \hbox{ is } m \hbox{ resonant for } j,$$
$\hbox{Res}$ the set multi-indices such that $H_1(s)^{i_1}\dots H_{n-1}(s)^{i_{n-1}} \in \mathbb{C}(\Sigma)$, $a_{j,i_1,\dots,i_{n-1}}\in\mathbb{C}(\Sigma)$ and $L_i$ logarithmic functions.
\end{prop}

Remark that the condition on $a_{j,m_1,\dots,m_{n-1}}(s_0)$ is non empty only when $Mon^0(NVE_1)$ is resonant.

\begin{proof}
We will first prove the existence of the formal series $\varphi_j,\;j=1\dots n-1$. We prove it by recurrence on the order $k=i_1+\dots+i_{n-1}$.
For $k=1$, we consider the solutions of the $NVE_1$. This gives
$$Z_{0,\dots,0,1,0,\dots,0}= c_jH_j \hbox{ with the } 1 \hbox{ in position } j$$
As the particular solution is $q_1=q_2=\dots =q_{n-1}=0$, this gives for the first order
$$\varphi_j(s,c_1,\dots,c_{n-1})= c_jH_j,\quad j=1\dots n-1$$

Let us now assume the formal series $\varphi_j,\;j=1\dots n-1$ exist at order $k-1$. We will now prove they exist at order $k$. The $\varphi$ at order $k-1$ we have induces a solution of the $NVE_{k-1}$ given by
$$Z_{m_1,\dots,m_n}(s)= \prod\limits_{j=1}^{n-1} \varphi_j(s,c_1,\dots,c_{n-1})^{m_j} \hbox{ mod } <(c^i)_{i_1+\dots+i_n=k}>$$
We now want to add terms of order $k$ to this solution to obtain a solution of the $NVE_k$. Remark that a solution of $NVE_k$ does not always leads to a series expansion of a solution of $X$ near $\Gamma$. This is due to the linearisation process made in constructing the higher variational equations.

Let us look at the variables $Z_{m_1,\dots,m_n}(s)$ with $m_1+\dots+m_n\geq 2$. Using the formula
$$Z_{m_1,\dots,m_n}(s)= \prod\limits_{j=1}^{n-1} \varphi_j(s,c_1,\dots,c_{n-1})^{m_j} \hbox{ mod } <(c^i)_{i_1+\dots+i_n=k+1}>$$
we see that the unknown terms in $c$ of order $k$ in $\varphi_j$ are not necessary to compute these expressions. This is because the total valuation in $c$ of the $\varphi_j,\; j=1\dots n-1$ is $1$, and as we have $m_1+\dots+m_n\geq 2$, any term in $c$ of order $k$ in $\varphi_j$ after expanding the product gives terms of order at least $k+1$, which disappear in the modulo.

We now have almost all the components of a solution of the $NVE_k$ which comes from a series expansion of a solution of $X$ near $\Gamma$. The only components we do not know are $Z_{0,\dots,0,1,0,\dots,0}$ with the $1$ in position $1$ to $n-1$. These component will give the expression of $\varphi_j$ at order $k$. They can be computed using variation of constants, giving the formula
$$Z_{0,\dots,0,1,0,\dots,0}=H_j(s) \sum\limits_{2\leq \mid i\mid \leq k} \int f_{j,i}(s)\frac{Z_i(s)}{H_j(s)} ds $$
with $f_{j,i}\in\mathbb{C}(\Sigma)$ coming from the computation of the series expansion of $X$ near $\Gamma$. Now knowing that $Z_m$ is a $\mathbb{C}(\Sigma)$ linear combination of $(c_1H_1(s))^{i_1}\dots (c_dH_{n-1}(s))^{i_{n-1}}$ with $2\leq \mid i\mid \leq k$, we have
$$Z_{0,\dots,0,1,0,\dots,0}=H_j(s) \sum\limits_{2\leq \mid i\mid \leq k} \int g_{j,i}(s)\frac{(c_1H_1(s))^{i_1}\dots (c_{n-1}H_{n-1}(s))^{i_{n-1}}}{H_j(s)} ds $$
with $g_{j,i}\in\mathbb{C}(\Sigma)$. Using $\varphi_j$ at order $k-1$, we already know the expression of $Z_{0,\dots,0,1,0,\dots,0}$ at order $k-1$ in $c$, giving
$$Z_{0,\dots,0,1,0,\dots,0}=\varphi_j(s)_{\mid \hbox{order }k-1}+H_j(s) \sum\limits_{\mid i\mid = k} \int g_{j,i}(s)\frac{(c_1H_1(s))^{i_1}\dots (c_{n-1}H_{n-1}(s))^{i_{n-1}}}{H_j(s)} ds.$$
The $VE_k$ has a virtually Abelian Galois group, thus $Z_{0,\dots,0,1,0,\dots,0}$ is in a virtually Abelian extension of $\mathbb{C}(\Sigma)$ for any $c$. Then each term of the sum should be as there cannot be any cancellation between the integrals due to each term having different coefficient $c^i$.

Let us now prove the two following Lemmas

\begin{lem}\label{lem1}
If a hyperexponential function $H\in \mathbb{C}(\Sigma)(H_1,\dots,H_{n-1})$ admits an integral in $\mathbb{C}(\Sigma)(H_1,\dots,H_{n-1})$ then it admits an integral in $\mathbb{C}(\Sigma).H$.
\end{lem}

\begin{proof}[Proof of the Lemma]
The integral of $H$ has at most an additive monodromy group over the base field $\mathbb{C}(\Sigma)(H)$. However, when know that it is in $\mathbb{C}(\Sigma)(H_1,\dots,H_{n-1})$ which has a multiplicative group. This implies that
$$\int H(s) ds \in \mathbb{C}(\Sigma)(H)$$
Now two cases. If $H\in\mathbb{C}(\Sigma)$, then the Lemma follows immediately as $\int H(s) ds \in \mathbb{C}(\Sigma)$ and then $\int H(s) ds/H \in \mathbb{C}(\Sigma)$.
If $H\notin\mathbb{C}(\Sigma)$, we write
$$\int H(s) ds= F(s,H(s)),\quad F\in\mathbb{C}(\Sigma)(x).$$
Now as $H$ is hyperexponential, we can act the Galois group on this equality giving
$$\int \alpha H(s) ds=F(s,\alpha H(s)),\quad \alpha\in\mathbb{C}^*.$$
Making a series expansion at $\alpha=0$ and identifying the powers of $\alpha$, we obtain 
$$\int H(s) ds= g(s) H(s),\;\;g\mathbb{C}(\Sigma)$$
which gives the Lemma.
\end{proof}

\begin{lem}\label{lem2}
If a hyperexponential function $H\notin \overline{\mathbb{C}(\Sigma)}$ admits an integral in virtually Abelian extension of $\mathbb{C}(\Sigma)$, then it admits an integral in $\mathbb{C}(\Sigma).H$.
\end{lem}

\begin{proof}[Proof of the Lemma]
The integral of $H$ has at most an additive monodromy group over the base field $\mathbb{C}(\Sigma)(H)$. As we also know that $H\notin \overline{\mathbb{C}(\Sigma)}$, the monodromy acts on $H$ multiplicatively
$$\sigma_{\alpha}(H)=\alpha H$$
If the monodromy group of $\int H$ over $\mathbb{C}(\Sigma)(H)$ is not identity, we have also a monodromy element
$$\delta_{\beta}(\int H)=\beta+\int H.$$
Now computing the commutator, we have
$$[\sigma_{\alpha},\delta_{\beta}]=\delta_{\beta (1-\alpha)}$$
So the monodromy would not be commutative, which is not compatible with the hypotheses. This implies that the monodromy is identity, and so that $\int H \in \mathbb{C}(\Sigma)(H)$. We now apply the previous Lemma \ref{lem1}, giving that it admits an integral in $\mathbb{C}(\Sigma).H$.
\end{proof}

Let us first assume we have the non resonance hypothesis. We have
$$\frac{(c_1H_1(s))^{i_1}\dots (c_{n-1}H_{n-1}(s))^{i_{n-1}}}{H_j(s)} \notin \mathbb{C}(\Sigma)$$
Due to the definition of $\Sigma$, we have moreover that
$$\frac{(c_1H_1(s))^{i_1}\dots (c_{n-1}H_{n-1}(s))^{i_{n-1}}}{H_j(s)} \notin \overline{\mathbb{C}(\Sigma)}$$
We can now use the Lemma \ref{lem2}, and we obtain a solution for $Z_{0,\dots,0,1,0,\dots,0}$ of the form
$$Z_{0,\dots,0,1,0,\dots,0}=\varphi_j(s)_{\mid \hbox{order }k-1}+\sum\limits_{\mid i\mid = k} \tilde{g}_{j,i}(s)(c_1H_1(s))^{i_1}\dots (c_{n-1}H_{n-1}(s))^{i_{n-1}} $$
This gives an expression for $\varphi_j$ at order $k$ in $c$ for $j=1\dots n-1$.

We now assume the hypothesis $Gal^0(NVE_k)\simeq Gal^0(NVE_1),\; \forall k\in\mathbb{N}^*$. We have to integrate terms of the form
$$g_{j,i}(s)\frac{(c_1H_1(s))^{i_1}\dots (c_{n-1}H_{n-1}(s))^{i_{n-1}}}{H_j(s)}$$
As the Galois group does not grow, we can apply Lemma \ref{lem1}, and so we know it admits an integral of the form
$$\tilde{g}_{j,i}(s)\frac{(c_1H_1(s))^{i_1}\dots (c_{n-1}H_{n-1}(s))^{i_{n-1}}}{H_j(s)},\quad \tilde{g}_{j,i}\in \mathbb{C}(\Sigma)$$
If the multi index $i$ is resonant with respect to $j$, we moreover have that
$$\frac{(c_1H_1(s))^{i_1}\dots (c_{n-1}H_{n-1}(s))^{i_{n-1}}}{H_j(s)} \in \mathbb{C}(\Sigma)$$
So we can freely add a constant to this integral, keeping the same form, just changing the $\tilde{g}_{i,j}(s)$. As $s_0$ is a regular point, the $H$ do not vanish nor become singular at $s_0$, the function $g_{i,j}$ is not singular at $s_0$, and thus adjusting the constant we can always assume that $\tilde{g}_{i,j}(s_0)=0$.
Now doing this on all terms, we obtain a solution for $Z_{0,\dots,0,1,0,\dots,0}$ of the form
$$Z_{0,\dots,0,1,0,\dots,0}=\varphi_j(s)_{\mid \hbox{order }k-1}+\sum\limits_{\mid i\mid = k} \tilde{g}_{j,i}(s)(c_1H_1(s))^{i_1}\dots (c_{n-1}H_{n-1}(s))^{i_{n-1}} $$
with moreover $\tilde{g}_{j,i}(s_0)=0$ when $i$ is resonant with respect to $j$.
This gives the expression for $\varphi_j$ at order $k$ in $c$ for $j=1\dots n-1$.

Now let us focus on the last case $\varphi_n$. We have
$$\frac{\partial}{\partial s}\varphi_n(s)=\frac{1}{X_n(s,\varphi_j(s)_{j=1\dots n-1})}$$
Expanding in series the right term, we obtain an expression of the form
$$\varphi_n(s)=\int \frac{1}{X_n(s,0)} ds +\sum\limits_{1\leq \mid i\mid \leq k} \int g_{n,i}(s)(c_1H_1(s))^{i_1}\dots (c_{n-1}H_{n-1}(s))^{i_{n-1}} ds$$
with $g_{n,i}\in\mathbb{C}(\Sigma)$. The set $\hbox{Res}$ is the the multi-indices such that
$$H_1(s)^{i_1}\dots H_{n-1}(s)^{i_{n-1}} \in \mathbb{C}(\Sigma).$$
Remark that if $i\notin \hbox{Res}$, then we moreover have
$$H_1(s)^{i_1}\dots H_{n-1}(s)^{i_{n-1}} \notin \overline{\mathbb{C}(\Sigma)}$$
as $\mathbb{C}(\Sigma)$ contains all algebraic functions in the Picard Vessiot field $\mathcal{P}$. Now for each $i\in Res$, we have that
$$g_{n,i}(s)(c_1H_1(s))^{i_1}\dots (c_{n-1}H_{n-1}(s))^{i_{n-1}} \in\mathbb{C}(\Sigma)$$
and thus its integral is a logarithmic function. And for $i\notin \hbox{Res}$, Lemma \ref{lem2} applies, and gives an integral of the form
$$\int g_{n,i}(s)(c_1H_1(s))^{i_1}\dots (c_{n-1}H_{n-1}(s))^{i_{n-1}} ds=\tilde{g}_{n,i}(s)(c_1H_1(s))^{i_1}\dots (c_{n-1}H_{n-1}(s))^{i_{n-1}}$$
This gives the expression of $\varphi_n$ of the Proposition.
\end{proof}

Let us remark that in the first case, we also proved that the Galois group does not grow, as we built the general solution $(\varphi_j(s,\cdot))_{j=1\dots n-1}$ as a formal series with coefficients in the Picard Vessiot field of the $NVE_1$. Thus the solutions of higher normal variational equations also belong to this field. This proves the corollary

\begin{cor}\label{cor1}
Under the conditions of Proposition \ref{propformal1}, $Mon^0(NVE_1)$ non-resonant implies $\hbox{Gal}^0(NVE_k) \simeq \mathbb{C}^{l-1} \;\forall k\in\mathbb{N}^*$.
\end{cor}

This also implies that Theorem \ref{thmmain2} implies the right to left implication of Theorem \ref{thmmain1}. The right to left implication of Theorem \ref{thmmain1} is given by Ayoul-Zung Theorem. So we still only have to prove Theorem \ref{thmmain2}. Proposition \ref{propformal1} is also close to a more streamlined equivalent of \cite{110} in the non Hamiltonian case with ``small'' Galois group. Indeed, our solutions series allow to find a linearisation map (see section 3.3.), which induces a linear transformation on the jet space diagonalizing the $NVE_k$.

\subsection{Formal integrability}

\begin{prop}\label{propformal2}
Let us consider $X$ a meromorphic vector field in the neighbourhood of an algebraic curve $\bar{\Gamma}$ with $\Gamma$ a solution of $X$. Let us assume that $Gal^0(VE_k)$ is Abelian $\forall k\in\mathbb{N}^*$ and note $l=dim(Gal^0(NVE_1))+1$. Assume $\hbox{Gal}^0(NVE_k) \simeq \mathbb{C}^{l-1} \;\forall k\in\mathbb{N}^*$. Then $X$ is formally $(l,n-l)$ integrable on a finite covering of a neighbourhood of $\bar{\Gamma}$.
\end{prop}

\begin{proof}
We can always assume the vector field $X$ to be gauge reduced as this does not impact integrability on a finite covering of a neighbourhood of $\Gamma$.
Let us now make the time reduction on $X$ and apply Proposition \ref{propformal1}. We build the formal series solution $\varphi$. The restricted function $(\varphi_j(s,\cdot))_{j=1\dots n-1}$ is at first order given by
$$\varphi_1=c_1H_1(s),\dots,\varphi_{n-1}=c_{n-1}H_{n-1}(s)$$
Thus we can formally invert the map
$$(\varphi_j)_{j=1\dots n-1}(s,\cdot) :(c_1H_1(s),\dots,c_{n-1}H_{n-1}(s)) \rightarrow \mathbb{C}^{n-1}$$
giving
$$c_jH_j(s)=\Phi_j(q,s),\qquad q\in \mathbb{C}^{n-1}$$
Now the hyperexponential solutions $H_1,\dots,H_{n-1}$ are possibly not algebraically independent. This is the case when $l<n$. In such case, we can build $n-l$ independent non trivial relations between the $H_j$ of the form
$$\prod\limits_{j=1}^{n-1} H_j(s)^{k_j}\in\mathbb{C}(\Sigma),\quad k_i\in\mathbb{Z}$$
Now replacing the $H_j$ by $\Phi_j(q,s)$, this allows to build formally $n-l$ first integrals $F_1,\dots,F_{n-l}\in\mathbb{C}(\Sigma)((q))$. They are functionally independent because they are independent at first order. Now first integrals are not affected by the time reduction, and thus $F_1,\dots,F_{n-l}$ are also formal first integrals of the vector field $X$.

We now want to build the formal vector fields. We consider the $n$ functions
$$J_i=\ln H_i(s)- \ln \Phi_i(q,s)\quad i=1\dots n-1$$
$$J_n=\sum\limits_{i\in \hbox{Res}}  \frac{\Phi_1(q,s)^{i_1}\dots \Phi_{n-1}(q,s)^{i_{n-1}}}{H_1(s)^{i_1}\dots H_{n-1}(s)^{i_{n-1}}} L_i(s)+$$
$$\sum\limits_{i\in\mathbb{N}^{n-1}}  a_{n,i_1,\dots,i_{n-1}}(s) \Phi_1(q,s)^{i_1}\dots \Phi_{n-1}(q,s)^{i_{n-1}}-t$$
These formal expressions are functionally independent first integrals of the vector field $X$ after time reduction, and thus so they are before the time reduction.

Let us consider the vector space $V$ of vectors $v\in \mathbb{C}(\Sigma)((q))^n$ such that
$$\sum\limits_{j=1}^{n-1} \partial_{q_j} F_i v_j+\partial_{s} F_i v_n=0 \quad \forall i=1\dots n-l$$
The commuting vector fields we are searching should be in $V$ as the $F_i$ should be first integrals of them. As the $F_i$ are independent, the dimension of $V$ is $l$ over $\mathbb{C}(\Sigma)((q))$, and we note $v_1,\dots,v_l$ a basis of them.

The logarithmic derivatives of the $F_i$ are linear combinations of the derivatives of the $J_i$. So let us consider a basis $B$ of the supplementary space of dimension $l$ of these linear combinations, and note
$$\tilde{J}_i=\sum\limits_{j=1}^n B_{i,j} J_j$$
Remark we can assume $\tilde{J}_n= J_n$ as $J_n$ cannot intervene in the expressions of the first integrals $F_i$ because of the $t$ appearing in it.  

Let us now define the $n\times n$ matrix
$$\hbox{Jac}=\left(\begin{array}{cccc} \partial_{q_1} J_1&\dots& \partial_{q_{n-1}} J_1 & \partial_{s} J_1\\ & \dots & & \\ \partial_{q_1} J_n&\dots& \partial_{q_{n-1}} J_n & \partial_{s} J_1\\ \end{array}\right)$$
and the $l\times l$ matrix
$$\tilde{\hbox{Jac}}=\left(\begin{array}{cccc} \mathcal{L}_{v_1} \tilde{J}_1&\dots&\mathcal{L}_{v_l} \tilde{J}_1\\ & \dots & \\ \mathcal{L}_{v_1} \tilde{J}_l &\dots& \mathcal{L}_{v_l} \tilde{J}_l\\ \end{array}\right)$$
All the logarithmic derivatives of $J_1,\dots,J_{n-1}$ are in $\mathbb{C}(\Sigma)((q))$, and thus so are $\mathcal{L}_{v_i} J_j,\; j=1\dots n-1, i=1\dots l$. For $J_n$, when computing $\mathcal{L}_{v_i} J_n$, the coefficients in front of the logarithmic functions $L_i(s)$ are first integrals, and thus their Lie derivative with respect to $v_i$ is zero. Thus the logarithmic functions $L_i(s)$ are always differentiated, and thus $\mathcal{L}_{v_i} \tilde{J}_n\in \mathbb{C}(\Sigma)((q)),\;\forall i=1\dots l$. So the coefficients of matrix $\tilde{\hbox{Jac}}$ are in $\mathbb{C}(\Sigma)((q))$.

\begin{lem}\label{lembog}
Let us consider an invertible matrix $M\in M_n(K)$ where $K$ is a differential field for the derivations in $x_1,\dots,x_n$. The vector fields $\sum_{i=1}^n M_{i,j} \frac{\partial}{\partial x_i}$ pairwise commute if and only if the $1$-forms $\sum_{j=1}^n (M^{-1})_{i,j} dx_j$ are closed.
\end{lem}

\begin{proof}[Proof of the Lemma]
In this proof, we will note $[\;]_j$ for the $j$-th column extraction of a matrix, and $(\;)_{k,j}$ the $k,j$ coefficient extraction. The closure condition for the differential forms writes
\begin{align*}
[\partial_{q_i} (M^{-1})]_j & =[\partial_{q_j} (M^{-1})]_i, \quad & \forall i,j  \\
[M^{-1}(\partial_{q_i}M)M^{-1}]_j & =[M^{-1}(\partial_{q_j}M)M^{-1}]_i, \quad & \forall i,j  \\
M^{-1}(\partial_{q_i}M)[M^{-1}]_j & =M^{-1}(\partial_{q_j}M)[M^{-1}]_i, \quad & \forall i,j  \\
(\partial_{q_i}M)[M^{-1}]_j & =(\partial_{q_j}M)[M^{-1}]_i, \quad & \forall i,j  \\
\sum\limits_{k=1}^n [\partial_{q_i}M)]_k(M^{-1})_{k,j} & = \sum\limits_{k=1}^n [\partial_{q_j}M]_k (M^{-1})_{k,i}, \quad & \forall i,j  \\
\sum\limits_{k=1}^n ((M^\intercal)^{-1})_{j,k} [\partial_{q_i}M)]_k & = \sum\limits_{k=1}^n ((M^\intercal)^{-1})_{i,k} [\partial_{q_j}M]_k, \quad & \forall i,j  \\
\sum\limits_{k=1}^n ((M^\intercal)^{-1})_{j,k} (\partial_{q_i}M))_{p,k} & = \sum\limits_{k=1}^n ((M^\intercal)^{-1})_{i,k} (\partial_{q_j}M)_{p,k}, \quad & \forall i,j,p  
\end{align*}
Noting 
$$B_p=\left(\begin{array}{ccc} (\partial_{q_1}M)_{p,1}& \dots &  (\partial_{q_n}M)_{p,1}\\ \dots & & \dots \\ (\partial_{q_1}M)_{p,n}& \dots &  (\partial_{q_n}M)_{p,n}\\  \end{array}\right),$$
this relation above rewrites
\begin{align*}
((M^\intercal)^{-1} B_p)_{i,j} & =((M^\intercal)^{-1} B_p)_{j,i}, \quad & \forall i,j,p 
(M^\intercal)^{-1} B_p & = B_p^\intercal M^{-1},\quad & \forall p.
\end{align*}
In other words the matrix $(M^\intercal)^{-1} B_p$ is symmetric for all $p$. Let us now write the commutation condition
\begin{align*}
\sum\limits_{k=1}^n (M)_{k,i} [\partial_{q_k} M]_j & = \sum\limits_{k=1}^n (M)_{k,j} [\partial_{q_k} M]_i,\quad & \forall i,j\\
\sum\limits_{k=1}^n (M)_{k,i} (\partial_{q_k} M)_{p,j} & = \sum\limits_{k=1}^n (M)_{k,j} (\partial_{q_k} M)_{p,i},\quad & \forall i,j,p\\
\sum\limits_{k=1}^n (M^\intercal)_{i,k}  (B_p^\intercal)_{k,j} & = \sum\limits_{k=1}^n (M^\intercal)_{j,k} (B_p^\intercal)_{k,i},\quad & \forall i,j,p\\
 (M^\intercal B_p^\intercal)_{i,j} & = (M^\intercal B_p^\intercal)_{j,i} ,\quad & \forall i,j,p  \\
 M^\intercal B_p^\intercal & = B_p M  ,\quad & \forall p.
\end{align*}
In other words the matrix $B_p M$ is symmetric for all $p$.

Now to conclude, multiplying the first condition by $M$ on the right and by $M^\intercal$ on the left, we obtain
$$(M^\intercal)^{-1} B_p = B_p^\intercal M^{-1} \Leftrightarrow B_p M= M^\intercal B_p^\intercal $$
\end{proof}

The matrix $\tilde{\hbox{Jac}}$ is a submatrix of $\hbox{Jac}$. The matrix $\hbox{Jac}$ is invertible and its lines form closed $1$-forms. So we can apply Lemma \ref{lembog} with $M=\hbox{Jac}^{-1}$, and thus the columns of $\hbox{Jac}^{-1}$ form commuting vector fields. The vectors
$$Y_j=\sum\limits_{i=1}^l (\tilde{\hbox{Jac}}^{-1})_{i,j} v_i,\quad j=1\dots l$$
are linear combinations of the columns of $\hbox{Jac}^{-1}$, and thus are commutative vector fields. 

Let us finally check that $X$ is among them. We have $\mathcal{L}_X J_i=0 ,\; \forall i=1\dots n-1$ and $\mathcal{L}_X J_n=1$. These equalities rewrite in matrix form
$$\hbox{Jac}\left(\begin{array}{c} X_1\\ \dots \\ X_n \end{array}\right)=\left(\begin{array}{c} 0\\ \dots \\ 0\\ 1 \end{array}\right)$$
and thus $X$ is given by the last column of $\hbox{Jac}^{-1}$. As moreover $\mathcal{L}_X F_i=0 ,\; \forall i=1\dots n-l$, $X$ belongs to the vector space $V$ and thus is a linear combination of the $Y_j$. After a basis change, we can assume for example that $Y_l=X$.

Thus $X$ admits $n-l$ formal first integrals $F_1,\dots,F_{n-l}\in\mathbb{C}(\Sigma)((q))$ and $l$ formal commuting vector fields $X,Y_1,\dots,Y_{l-1}$ with coefficients in $\mathbb{C}(\Sigma)((q))$. Thus $X$ is formally $(l,n-l)$ integrable on a finite covering of a neighbourhood of $\bar{\Gamma}$.
\end{proof}

The first integrals come from the algebraic relations between the hyperexponential functions, and can effectively appear because the non resonance condition only holds on $\mathbb{N}$ in contrary to the relations taken into account by the Galois group which hold over $\mathbb{Z}$. The Proposition \ref{propformal2} gives a formal result for Theorem \ref{thmmain2}, and using Corollary \ref{cor1}, also gives a formal result for the right to left implication of Theorem \ref{thmmain1}.\\

\noindent
\textbf{Example}
$$\dot{q}=\frac{\alpha q}{s(q^3+q^2s+s)},\dot{s}=\frac{1}{q^3+q^2s+s}$$
We make the time reduction, giving
$$\dot{q}=\frac{\alpha q}{s},\dot{t}=q^3+q^2s+s$$
where $\;\dot{}\;$ is now the derivation in $s$. The series of Proposition \ref{propformal1} are
$$q=c_1s^{\alpha} ,\quad t=\frac{1}{2}s^2+\frac{s^2\left(c_1s^{\alpha}\right)^2}{2\alpha+2} +\frac{s\left(c_1s^{\alpha}\right)^3}{3\alpha+1} $$
There are no resonant terms for $\alpha\neq -1,-1/3$. The Galois group for $\alpha\notin \mathbb{Q}$ of $NVE_1$ is $\mathbb{C}^*$, and so there are no meromorphic first integrals in $q,s$. The functions $J_i$ are
$$J_1=\alpha \ln s- \ln q,\quad J_2=\frac{1}{2}s^2+\frac{s^2\left(c_1s^{\alpha}\right)^2}{2\alpha+2} +\frac{s\left(c_1s^{\alpha}\right)^3}{3\alpha+1}-t$$
As there are no first integrals, we have $\tilde{\hbox{Jac}}=\hbox{Jac}$, and we obtain
$$\hbox{Jac}=\left(\begin{array}{cc} 
-\frac{1}{q} & \frac{\alpha}{s}\\
\frac{2s^2q}{2\alpha+2}+\frac{3sq^2}{3\alpha+1} & s+\frac{2sq^2}{2\alpha+2}+\frac{q^3}{3\alpha+1}
 \end{array}\right)$$
Now inverting this matrix, we find for the first column up to multiplication by a constant the following commuting vector field
$$\frac{(3\alpha q^2s+3\alpha^2s+q^2s+4\alpha s+s+\alpha q^3+q^3)q}{q^3+q^2s+s} \frac{\partial}{\partial q}-
\frac{q^2s(3\alpha q+3\alpha s+3q+s)}{q^3+q^2s+s} \frac{\partial}{\partial s}$$
When $\alpha\in\mathbb{Q}$, the system admits a first integral,
$$F_1(q,s)=q^{\hbox{denom}(\alpha)}s^{-\hbox{numer}(\alpha)}$$
This first integral is indeed the exponential of $-\hbox{denom}(\alpha) J_1$.\\

Remark that when $\dot{s}=1$ already before the time reduction, the equation in time is $\dot{t}=1$. Then the matrix $\hbox{Jac}$ is of the form
$$\left(\begin{array}{ccc|c} & & & *\\ & -\nabla_q \Phi & & * \\  & & & *\\\hline  0& \dots & 0 & * \end{array} \right)$$
The vector space $V$ of vector fields independent of $X$ such that with respect to them the Lie derivative of the first integrals of $X$ is $0$ is sent by this matrix on $\hbox{LieGal}(NVE_1)$. Thus $\hbox{Jac}^{-1}.\hbox{LieGal}(NVE_1)$ is the vector space generated by commuting vector fields independent with $X$ and with common first integrals the $F_i$. Now this can also be seen as the the action of $\hbox{LieGal}(NVE_1)$ on $\varphi$. Indeed, differentiating the $\varphi$ along $\hbox{LieGal}(NVE_1)$ is simply
\begin{align*}
(\nabla_q \Phi^{-1})(\Phi(q,s)).\hbox{LieGal}(NVE_1) & =-(\nabla_q \Phi)^{-1}.\hbox{LieGal}(NVE_1)\\
 & =\hbox{Jac}^{-1}.\hbox{LieGal}(NVE_1)
\end{align*}
the last equality being true because $\hbox{Jac}$ is a block triangular matrix.
This process is easier to compute and show how $\hbox{LieGal}(NVE_1)$ acts on the non linear system.

\subsection{Formal linearisability}

Let us now find a formal coordinates change for Theorem \ref{thmmain3}.

\begin{prop}\label{propformal3}
Let $X$ be a time dependant vector field meromorphic on an algebraic finite covering $\mathcal{C}$ of a neighbourhood of $\{0\in\mathbb{C}^{n-1}\}\times\mathbb{P}$. Let us note $\pi:\mathcal{C} \mapsto \mathbb{C}^{n-1}$ the projection, $\Gamma=\pi^{-1}(0)\setminus S$ where $S$ is the finite set of singular points of $X$ on $\pi^{-1}(0)$ and $\bar{\Gamma}=\pi^{-1}(0)$. Assume $X=0$ on $\Gamma$, the $NVE_1$ near $\bar{\Gamma}$ is Fuchsian and $Mon^0(NVE_1)$ is diagonal and Diophantine. If $\hbox{Gal}^0(NVE_k) \simeq \mathbb{C}^{l-1}$ for all $k\in \mathbb{N}^*$ then the vector field $X$ is formally linearisable on a neighbourhood of $\bar{\Gamma}$.
\end{prop}

\begin{proof}
Adding the equation $\dot{t}=1$ with $\dot{}$ corresponding to the derivative in $s$, the vector field $X(s,q)$ defines a system time reduced of the from \eqref{eqreduc}.
We make a gauge reduction of $X$ (the vector field $X$) and apply Proposition \ref{propformal1}. We can formally invert the map
$$(\varphi_j)_{j=1\dots n-1}(s,\cdot) :(c_1H_1(s),\dots,c_{n-1}H_{n-1}(s)) \rightarrow \mathbb{C}^{n-1}$$
giving
$$c_jH_j(s)=\Phi_j(q,s),\qquad q\in \mathbb{C}^{n-1}$$
Thus in new coordinates defined by $\Phi$, the vector field $X$ becomes linear with associated matrix
\begin{equation}\label{eqlin}
\left(\begin{array}{cccc}
\tilde{X}_{1,1}(s)&0& \dots & 0\\
 & \dots & &  \\
0&\dots &0& \tilde{X}_{n-1,n-1}(s)\\
\end{array}\right)
\end{equation}
where $\tilde{X}_{i,i}(s)$ are logarithmic derivatives of the $H_i$. The coordinates change $\Phi$ is a formal series with coefficients in $\mathbb{C}(\Sigma)$.

Let us note $P$ the matrix with coefficients in $\mathbb{C}(\Sigma)$ given by $\Phi$ at first order. We now consider the transformation $P^{-1}\Phi$. The matrix $P^{-1}$ is a gauge transformation sending equation \eqref{eqlin} to the linear part of $X$. Moreover by construction, the application $P^{-1}\Phi$ is tangent to identity.

We know that $P^{-1}\Phi \in \mathbb{C}(\Sigma)[[q_1,\dots,q_{n-1}]]^{n-1}$. Let us now consider the action of $G=\hbox{Gal}(\mathbb{C}(\Sigma):\mathbb{C}(\bar{\Gamma}))$ on it. As the coefficients of the linear part of $X$ are in $\mathbb{C}(\bar{\Gamma})$ (recall that $X$ is meromorphic in a neighbourhood of $\bar{\Gamma}$), the linear part of $X$ is left invariant by the action of $G$. Now acting $G$ on $P^{-1}\Phi$ gives possibly several transformations, all tangent to identity, sending $X$ to its linear part. Composing such a transformation with the inverse of another one, we obtain a transformation stabilizing the linear vector field associated to the linear part of $X$, and tangent to identity. Such a transformation has to be identity.

Thus $G$ leaves $P^{-1}\Phi$ invariant, and thus $P^{-1}\Phi \in \mathbb{C}(\bar{\Gamma})[[q_1,\dots,q_{n-1}]]^{n-1}$. Thus $X$ is conjugated to its linear part by a transformation in $\mathbb{C}(\bar{\Gamma})[[q_1,\dots,q_{n-1}]]^{n-1}$, i.e. a formal transformation on a neighbourhood of $\bar{\Gamma}$.
\end{proof}

\section{The Ziglin group}

\subsection{Definitions}

\begin{defi}
Let us consider $X$ a gauge reduced vector field and a point $s_0\in\Gamma$. Let us consider a closed curve $\gamma$ on $\Gamma$ with $s_0\in\gamma$ and note $\Phi_\gamma\in \mathbb{C}\{q_1,\dots,q_{n-1}\}^{n-1}$ the germ of holomorphic map given by the flow of $X$ along $\gamma$ with initial condition $s_0,q_1,\dots,q_{n-1}$. The Ziglin group $\hbox{Zig}(X)$ is the group of germs of holomorphic maps $\Phi_\gamma$ for all such curves $\gamma$. The subgroup $\hbox{Zig}^0(X)$ is the group of holomorphic maps $\Phi_\gamma$ for all such curves $\gamma$ whose lift on $\Sigma$ is closed.
\end{defi}

The Jacobian matrices of the elements of the Ziglin group are the monodromy matrices of the $NVE_1$. The Jacobian matrices of the elements of $\hbox{Zig}^0(X)$ are monodromy matrices of the $NVE_1$ along closed curves on $\Sigma$, i.e. elements of $\hbox{Mon}^0(NVE_1)$. Due to this, the Ziglin group generalize the monodromy group of the $NVE_1$, which was first used by Ziglin \cite{79} to prove non integrability.

\begin{prop}
Let us consider $X$ a time and gauge reduced vector field. If $Gal^0(VE_k)$ is Abelian $\forall k\in\mathbb{N}^*$, then $\hbox{Zig}^0(X)$ is Abelian. Moreover
$$\hbox{Zig}(X)/\hbox{Zig}^0(X) \simeq \hbox{Gal}(NVE_1)/\hbox{Gal}^0(NVE_1)$$
\end{prop}

\begin{proof}
Let us consider the flow $\Phi$ with initial condition $q_1^0,\dots,q_{n-1}^0$ at $s_0$. Its series expansion at order $k$ in the initial conditions $q_1^0,\dots,q_{n-1}^0$ gives a vector of functions in $s$, which is a solution of the $NVE_k$. Now computing this flow along a closed loop $\gamma$ on $\Sigma$ defines an element of $\hbox{Zig}^0(X)$. Doing the same on the series expansion at order $k$ defines a monodromy matrix of the $NVE_k$. Now knowing that $\gamma$ is closed on $\Sigma$, we also know that this monodromy matrix belongs to $\hbox{Gal}^0(NVE_k)$. This group is Abelian by hypothesis, and thus any pair of elements of $\hbox{Zig}^0(X)$ commute up to order $k$. This is valid for any $k\in\mathbb{N}^*$, and the elements of $\hbox{Zig}^0(X)$ are holomorphic maps. Thus $\hbox{Zig}^0(X)$ is Abelian.

Let us now remark that $\hbox{Zig}(X)/\hbox{Zig}^0(X)$ is a subgroup of the group of the covering $\Sigma$ over $\bar{\Gamma}$ (which is a finite group as $\Sigma$ is an algebraic Riemann surface over $\bar{\Gamma}$). By construction of $\Sigma$, we also have
$$\hbox{Gal}(\mathbb{C}(\Sigma):\mathbb{C}(\bar{\Gamma}))=\hbox{Gal}(NVE_1)/\hbox{Gal}^0(NVE_1)$$
Thus 
$$\hbox{Zig}(X)/\hbox{Zig}^0(X) \subset \hbox{Gal}(NVE_1)/\hbox{Gal}^0(NVE_1)$$

Now remark that $\Phi_\gamma$ at first order gives the element of $\hbox{Mon}(NVE_1)$ generated by the closed curve $\gamma$. And thus
$$\hbox{Zig}(X)/\hbox{Zig}^0(X) \supset \hbox{Mon}(NVE_1)/\hbox{Mon}^0(NVE_1)$$
As we know that $\hbox{Mon}(NVE_1)/\hbox{Mon}^0(NVE_1)$ is finite, we deduce it is equal to $\hbox{Gal}(NVE_1)/\hbox{Gal}^0(NVE_1)$, which gives the Proposition.
\end{proof}

Remark that when computing $\hbox{Gal}(NVE_k),\; k\geq 2$, only integrals are necessary, and never algebraic extensions. Thus the number of connected components of $\hbox{Gal}(NVE_k)$ does not grows and so
$$\hbox{Gal}(NVE_k)/\hbox{Gal}^0(NVE_k) \simeq \hbox{Gal}(NVE_1)/\hbox{Gal}^0(NVE_1)$$
The Proposition thus extends this result to $\hbox{Zig}(X)$.

\subsection{Convergence}

\begin{prop}\label{propconv}
The formal series solution of Proposition \ref{propformal1} under the additional condition that $\hbox{Mon}^0(NVE_1)$ is Diophantine is convergent on a neighbourhood of $c=0$ with $s$ not projecting on a singularity of $X$.
\end{prop}

\begin{proof}
Let us fix an $s_0\in\Sigma$ which projects not on a singularity of $X$, and let us consider the restricted function $(\varphi_j(s_0,\cdot))_{j=1\dots n-1}$. This defines an invertible formal map in the neighbourhood of $(q_1,\dots,q_{n-1})=0$. Let us note the new coordinates $c_1,\dots,c_{n-1}$.

Now computing $(\varphi_j(s_0,\cdot))_{j=1\dots n-1}$ along a closed curve on $\Sigma$ with base point $s_0$ gives in the coordinates $c_1,\dots,c_{n-1}$ a diagonal transformation (recall that $\hbox{Mon}^0(NVE_1)$ is a diagonal multiplicative group). Thus in the coordinates $c_1,\dots,c_{n-1}$, the Ziglin group element $\Phi_\gamma$ is a linear diagonal transformation. The same hold for any closed curve $\gamma$ on $\Sigma$, and thus any element of $\hbox{Zig}^0(X)$ is diagonal in the coordinates $c_1,$ $\dots,c_{n-1}$.

Thus the coordinates change $(\varphi_j(s_0,\cdot))_{j=1\dots n-1}$ makes a simultaneous linearisation of all elements of $\hbox{Zig}^0(X)$. The linear part of $\hbox{Zig}^0(X)$ are the monodromy matrices $\hbox{Mon}^0(NVE_1)$ which are Diophantine by hypothesis. We can now apply Stolovitch Theorem 2.1 \cite{106}. The elements of $\hbox{Zig}^0(X)$ are formally linearisable, and thus holomorphically linearisable. More importantly, the linearisation map is unique provided that resonant monomials in the transformation vanish. This is indeed the case of $(\varphi_j(s_0,\cdot))_{j=1\dots n-1}$ as given by Proposition \ref{propformal1}. Thus by uniqueness, the transformation $(\varphi_j(s_0,\cdot))_{j=1\dots n-1}$ is convergent on a neighbourhood of $0$.

For now, we just have proved convergence at $s_0$. Let us note $s_1\in\Sigma$ which projects on a non singular point for $X$ and $\gamma$ a path going from $s_0$ to $s_1$. Let us consider $\Phi_\gamma$ the flow of $X$ along $\gamma$. We have then
$$(\varphi_j(s_1,\cdot))_{j=1\dots n-1}=\Phi_\gamma \circ (\varphi_j(s_0,\cdot))_{j=1\dots n-1}$$
As $\Phi_\gamma$ and $(\varphi_j(s_0,\cdot))_{j=1\dots n-1}$ are holomorphic in a neighbourhood of $0$, so is\\ $(\varphi_j(s_1,\cdot))_{j=1\dots n-1}$. Using the formula 
$$\frac{\partial}{\partial s}\varphi_n(s)=\frac{1}{X_n(\varphi_j(s)_{j=1\dots n-1},s)}$$
we deduce that $\varphi_n(s)$ also converges.
\end{proof}

Remark that the value of $\varphi(s_1,\cdot)$ depends on the path $\gamma$ chosen. Changing the path between $s_0$ to $s_1$ is equivalent to compose with an element of $\hbox{Zig}^0(X)$, which is diagonal in the coordinates $c_1,\dots,c_{n-1}$. Thus changing the path does not affect the convergence in a neighbourhood of $0$ but can change the size of this neighbourhood.

\begin{proof}[End of the proof of Theorem \ref{thmmain2}]
We now apply the same reasoning of Proposition \ref{propformal2}, but using now converging series. This gives us vector fields and first integrals as converging series. Now these are converging outside the singularities of the time and gauge reduced vector field. These singularities correspond to zeros and singularities of the initial vector field, and singularities of the parametrization. So when going back to the initial variables $x_1,\dots,x_n$, the singularities of the parametrization disappear and we obtain commuting vector fields and first integrals, meromorphic on a finite covering of a neighbourhood of $\Gamma$.  Finally, The number of vector fields produced is exactly $\hbox{dim}(\hbox{Gal}^0(NVE_1))+1$, as given by Proposition \ref{propformal2}.
\end{proof}

\subsection{End of proof of Theorem \ref{thmmain3}}

\begin{proof}
For the right to left implication, let us first remark that if the set of singularities $S$ is the whole $\Gamma$, then the Theorem is empty. If it is not the whole $\Gamma$, as $X$ is meromorphic on $\bar{\Gamma}$, then $S$ is finite. So we can use Proposition \ref{propformal3} and we have a formal linearisation. Using moreover Proposition \ref{propconv}, the formal series of Proposition \ref{propformal1} are convergent, and thus so is the coordinate change for the linearisation of $X$ of Proposition \ref{propformal3}.

Let us now prove the left to right implication. The system is equivalent to its linear part with a holomorphic variable change on a finite covering of a neighbourhood of $\Gamma$. In these coordinates, $X$ becomes linear, and as $\hbox{Gal}^0(NVE_1)$ is diagonal, the linear system can be furthermore gauge reduced, giving a system of the form (the time being noted $s$)
$$q'=\left(\begin{array}{cccc}
\tilde{X}_{1,1}(s)&0& \dots & 0\\
 & \dots & &  \\
0&\dots &0& \tilde{X}_{n-1,n-1}(s)\\
\end{array}\right)q$$
with $\tilde{X}_{i,i}(s)$ meromorphic on a finite covering of $\mathbb{P}^1$. The group $\hbox{Gal}^0(NVE_k)$ of such equation is $(\mathbb{C}^*)^{l-1}$ for all $k\in\mathbb{N}^*$ and the group $\hbox{Mon}^0(NVE_k)$ is multiplicative.
Now the gauge reduction is a variable change holomorphic on a finite covering of a neighbourhood of $\Gamma$, and so changes the monodromy group of $NVE_k$ with at most a finite extension (or a finite index subgroup). Thus the monodromy group of $NVE_k$ in the original coordinates is a finite extension of a diagonal group of matrices. As the system is Fuchsian, its Zariski closure is the Galois group, and thus $\hbox{Gal}^0(NVE_k)\simeq (\mathbb{C}^*)^{l-1}$ with the same integer $l$.
\end{proof}

\section{Completion and finite coverings}

Although the Theorems \ref{thmmain1},\ref{thmmain2} are ``in spirit'' inverse of the Ayoul-Zung Theorem, they are not exactly because the integrability produced is not exactly the same as in the original Ayoul-Zung Theorem
\begin{itemize}
\item The vector fields produced are defined on a finite covering over a neighbourhood of $\Gamma$, and thus are multivalued. This is due to algebraic extensions made for $\Sigma$. Thus strictly speaking, these vector fields are not meromorphic near $\Gamma \subset \mathbb{P}^n$.
\item The vector field is defined near $\bar{\Gamma}$ and Galoisian conditions are over $\bar{\Gamma}$, however the first integrals and vector fields produced are possibly not defined on a neighbourhood of $\bar{\Gamma}\setminus \Gamma$.
\end{itemize}

\subsection{Minimal finite coverings for integrability}

The finite covering problem can be understood at the linear level. The Theorem \ref{thmmain3} gives a linearisation without needing a finite covering, but the resulting matrix system is not diagonal. The construction of vectors fields using Proposition \ref{propformal2} however require gauge reduction, so diagonalization of the linear part.

\begin{prop}\label{propcover}
Let $X$ be a meromorphic vector field of a neighbourhood of $\bar{\Gamma}$, with $\Gamma$ an algebraic solution of $X$. Assume $NVE_1$ is Fuchsian, $\hbox{Gal}^0(NVE_k) \simeq \mathbb{C}^{l-1}$ for all $k\in \mathbb{N}^*$ and $Mon^0(NVE_1)$ is Diophantine. Let us note $\mathcal{J}_\Sigma,\mathcal{J}_{\bar{\Gamma}}$ the fields of first integrals of the $NVE_1$ with coefficients in $\mathbb{C}(\Sigma)$ and $\mathbb{C}(\bar{\Gamma})$ respectively.
The vector field is integrable on a neighbourhood of $\Gamma$ if and only if $\mathcal{J}_\Sigma/\mathcal{J}_{\bar{\Gamma}} \simeq \mathbb{C}(\Sigma)/\mathbb{C}(\bar{\Gamma})$.
\end{prop}

\begin{proof}
We can assume the vector field $X$ to be reduced as equation \eqref{eqreduc0}, as it does not require algebraic extensions (the base field stays the rational functions on $\bar{\Gamma}$).
We begin by the right to left implication. We only need to find $l-1$ commuting vector fields and $n-l$ first integrals (depending on the new time $s$). Let us first remark that the unitary minimal polynomial of an element of $\mathcal{J}_\Sigma$ has coefficients in $\mathcal{J}_{\bar{\Gamma}}$ (i.e. an algebraic first integral has a unitary minimal polynomial whose coefficients are also first integrals).

Let us note $w$ a generator of the algebraic extension $\mathbb{C}(\Sigma)/\mathbb{C}(\bar{\Gamma})$. We note $w_1,\dots,w_p$ its conjugates. By hypothesis, there exists $F\in \mathcal{J}_\Sigma$ such that the action of $G=\hbox{Gal}(\mathbb{C}(\Sigma):\mathbb{C}(\bar{\Gamma}))$ on $w$ and $F$ are exactly the same. We also note $F_1,\dots,F_p$ the conjugates of $F$. The Galois group $G$ acts as permutations on the $w_1,\dots,w_p$ and $F_1,\dots,F_p$. Let us note $Y_1,\dots,Y_{l-1}$ the vector fields obtained by Theorem \ref{thmmain2}. Let us note $(Y_{i,j})_{j=1\dots p}$ the conjugates of $Y_i$ (remark they could be dependent). We now consider the vector fields
$$\left(\sum\limits_{j=1}^p F_j^m Y_{i,j}\right)_{m=0\dots p-1,\; j=1\dots l-1}$$
Acting an element $\sigma\in G$ on such element gives
\begin{align*}
\sigma\left(\sum\limits_{j=1}^p F_j^m Y_{i,j}\right) =\sum\limits_{j=1}^p \sigma(F_j^m Y_{i,j}) = \sum\limits_{j=1}^p F_{\tau(j)}^m Y_{i,\tau(j)} =\sum\limits_{j=1}^p F_j^m Y_{i,j}
\end{align*}
with $\tau\in S_p$. So these vector fields have coefficients in $\mathbb{C}(\bar{\Gamma})$. The dimension of the vector space generated by $(Y_{i,j})_{j=1\dots p,i=1\dots l-1}$ is at least $l-1$. We acted on it a matrix block diagonal whose blocks are the Vandermonde matrix given by the $F_i$. As the $F_i$ are all different, this Vandermonde matrix is invertible, and the the dimension of
$$\left(\sum\limits_{j=1}^p F_j^m Y_{i,j}\right)_{m=0\dots p-1,\; j=1\dots l-1}$$
is at least $l$.

To conclude, we need to build also the suitable first integrals. These are the elements of $\mathcal{J}_{\bar{\Gamma}}$. By Theorem \ref{thmmain2}, we know there are $n-l$ independent first integrals in $\mathcal{J}_\Sigma$, i.e. $\mathcal{J}_\Sigma$ is of transcendence degree $n-l$. As $\mathcal{J}_{\bar{\Gamma}}$ is a subfield of $\mathcal{J}_\Sigma$ of finite index, it has the same transcendence degree and so also contains $n-l$ first integrals.

Let us now prove the left to right implication. Let us first remark that the inclusion
$$\mathcal{J}_\Sigma/\mathcal{J}_{\bar{\Gamma}} \subset \mathbb{C}(\Sigma)/\mathbb{C}(\bar{\Gamma})$$
is immediate. So we only have to prove the other way. We use Theorem \ref{thmmain2} to build vector fields $Y_1,\dots,Y_{l-1},Y_l=X$ and first integrals $I_1,\dots,I_{n-l}$. By hypothesis, we have $\tilde{Y}_1,\dots,\tilde{Y}_{l-1}$ meromorphic independent commuting vector fields in a neighbourhood of $\Gamma$. So this implies we can write
\begin{equation}\label{eqvec}
\tilde{Y}_i=\sum\limits_{j=1}^{l-1} f_{i,j} Y_j \quad i=1\dots l
\end{equation}
with $f_{i,j}$ are first integrals in $\mathcal{J}_\Sigma$. We can invert this linear relation, expressing the $Y_j$ in functions of the $\tilde{Y}$ and the first integrals.

Let us consider an element $\sigma\in \hbox{Gal}(\mathbb{C}(\Sigma):\mathbb{C}(\bar{\Gamma}))$, and assume that $\sigma$ fixes all the first integrals in $\mathcal{J}_{\Sigma}$. We want to prove that $\sigma=id$. Due to the above relation, as $\sigma$ fixes the $\tilde{Y}$ by assumption, we have that $\sigma$ fixes the $Y_j$. Let us now consider the system of equations
\begin{equation}\label{eqinv}
\mathcal{L}_{Y_j} f_i=\delta_{i,j} \quad i,j=1\dots l
\end{equation}
Now recalling the proof of Proposition \ref{propformal2}, we know solutions for this system as linear combinations
$$J_i=\ln H_i(s)- \ln \Phi_i(q,s)\quad i=1\dots n-1$$
$$J_n=\sum\limits_{i\in \hbox{Res}}  \frac{\Phi_1(q,s)^{i_1}\dots \Phi_{n-1}(q,s)^{i_{n-1}}}{H_1(s)^{i_1}\dots H_{n-1}(s)^{i_{n-1}}} L_i(s)+$$
$$\sum\limits_{i\in\mathbb{N}^{n-1}}  a_{n,i_1,\dots,i_{n-1}}(s) \Phi_1(q,s)^{i_1}\dots \Phi_{n-1}(q,s)^{i_{n-1}}-t$$
Let us note $\tilde{J}_1,\dots,\tilde{J}_{l-1},J_n$ such solutions of system \eqref{eqinv} ($J_n$ is a solution for $i=l$ as we noted $Y_l=X$). The exponentials of certain linear combinations of the $J_i$ form the first integrals $F_i$. These $F_i$ are the first integrals of the $Y_j$, and the common kernel of the $\mathcal{L}_{Y_j}$ is the functions in $F_1,\dots, F_{n-l}$.

Thus all the solutions of system \eqref{eqinv} are of the form
$$f_i=\tilde{J}_i+ \Psi_i(F_1,\dots,F_{n-l}),\quad i=1\dots l$$
Let us now look at the action of $\sigma$ on $\tilde{J}_i$. As $\sigma(\tilde{J}_i)$ should be still a solution of \eqref{eqinv}, we have
$$\sigma(\tilde{J}_i)=\tilde{J}_i+ \Psi_i(F_1,\dots,F_{n-l}).$$
The element fixes the first integrals $F_i$ (whose logs are linear combinations of the $J_i$), and thus the action of $\sigma$ on the $J_i$ is of the form
$$\sigma(J_i)=J_i+ \tilde{\Psi}_i(F_1,\dots,F_{n-l}),\quad i=1\dots n.$$
So in other words, the action of $\sigma$ on $H_i(s)/\Phi_j(q,s)$ is a multiplication by an arbitrary function of $(F_1,\dots,F_{n-l})$.

Now $\sigma$ can also be seen as a element on $\hbox{Gal}(NVE_1)$, and so can be represented by an $(n-1)\times (n-1)$ matrix. Let us precise the possible structures of $\hbox{Gal}(NVE_1)$. The group $\hbox{Gal}^0(NVE_1)$ is constituted of diagonal matrices.
Let us regroup the $H_i$ by blocks such that on each block any matrix of $\hbox{Gal}^0(NVE_1)$ is a multiple of identity. Then an element of finite order of $\hbox{Gal}(NVE_1)$ can
\begin{itemize}
\item Act as a finite group on a block.
\item Permute different blocks.
\end{itemize}
As $\sigma \in \hbox{Gal}(NVE_1)$, we have that $\sigma(H_i(s)/\Phi_i(q,s))$ should stay in the same differential field. Thus the action of $\sigma$ is a multiplication by an element of $\mathcal{J}_{\Sigma}$. Thus $\hbox{Gal}^0(NVE_1)$ acts the same after and before $\sigma$, and so $\sigma$ does not permute different blocks.

If $H_i$ and $H_j$ belong to the same block, then
$$\frac{H_i(s)\Phi_j(q,s)}{H_j(s)\Phi_i(q,s)}$$
is a first integral. Thus the action of $\sigma$ on it is identity. Thus $\sigma$ can at most multiply the $H_i(s)/\Phi_i(q,s)$ by a constant, and the same one on a block.

So the matrix associated to $\sigma$ is diagonal, and using that it fixes the first integrals, this diagonal matrix has to be in $\hbox{Gal}^0(NVE_1)$. However an element $\hbox{Gal}^0(NVE_1)$ acts trivially on $\mathbb{C}(\Sigma)$, and thus $\sigma=id$. 
\end{proof}

Any matrix of $\hbox{Gal}(NVE_1)$ commuting with $\hbox{Gal}^0(NVE_1)$ will act trivially on the vector fields $Y_i$. So we can consider the maximal subgroup of $\hbox{Gal}(NVE_1)$ commuting with $\hbox{Gal}^0(NVE_1)$ and then the quotient of $\hbox{Gal}(NVE_1)$ by it. This defines a finite group which is the Galois group of $\mathbb{C}(\tilde{\Sigma})$, the field generated by the coefficients of $Y_i$. Only the ``non-commutative'' part of $\hbox{Gal}(NVE_1)$ produce ramifications for the vector fields $Y_i$, and thus $\mathbb{C}(\tilde{\Sigma})$ is typically smaller than $\mathbb{C}(\Sigma)$.\\

\noindent
\textbf{Example 1}\\
$$\dot{q}_1=\alpha q_2,\; \dot{q}_2=\frac{\alpha q_1-s q_2}{s^2+1}$$
The Galois group is virtually diagonal, and after gauge reduction, we diagonalize it and we find the solutions
$$c_1(s+\sqrt{1+s^2})^\alpha,\; c_2(s-\sqrt{1+s^2})^\alpha.$$
So $\mathbb{C}(\Sigma)=\mathbb{C}(s,\sqrt{1+s^2})$. The only first integral with coefficients in $\mathbb{C}(\Sigma)$ is $q_1^2-(1+s^2)q_2^2$. Thus $\mathcal{J}_\Sigma=\mathcal{J}_{\bar{\Gamma}}$ and so this system is not integrable on a neighbourhood of $q=0$. It needs a $2$-covering for its commuting vector field.\\

\noindent
\textbf{Example 2}\\
$$\dot{q}_1=\alpha q_2+\frac{s q_1}{2(1+s^2)},\; \dot{q}_2=\frac{\alpha q_1-\frac{1}{2}s q_2}{s^2+1}$$
The Galois group is virtually diagonal, and after gauge reduction, we diagonalize it and we find the solutions
$$c_1(1+s^2)^{1/4}(s+\sqrt{1+s^2})^\alpha,\; c_2(1+s^2)^{1/4}(s-\sqrt{1+s^2})^\alpha.$$
We have $\mathbb{C}(\Sigma)=\mathbb{C}(s,\sqrt{1+s^2})$. The only first integral with coefficients in $\mathbb{C}(\Sigma)$ is $(q_1^2-(1+s^2)q_2^2)/\sqrt{1+s^2}$. Thus $\mathcal{J}_\Sigma/\mathcal{J}_{\bar{\Gamma}}\simeq \mathbb{C}(s,\sqrt{1+s^2})/\mathbb{C}(s)$ and so this system is integrable on a neighbourhood of $q=0$. The vector field and first integral produced by Theorem \ref{thmmain2} are
$$q_2\sqrt{1+s^2} \frac{\partial}{\partial q_1}+\frac{q_1}{\sqrt{1+s^2}} \frac{\partial}{\partial q_2},\quad \frac{q_1^2-(1+s^2)q_2^2}{\sqrt{1+s^2}}$$
We can remove the square root of the vector field by multiplying it with the first integral, and square the first integral
$$q_2(q_1^2-(1+s^2)q_2^2) \frac{\partial}{\partial q_1}+\frac{q_1(q_1^2-(1+s^2)q_2^2)}{1+s^2} \frac{\partial}{\partial q_2},\quad \frac{(q_1^2-(1+s^2)q_2^2)^2}{1+s^2}$$

\noindent
\textbf{Example 3}\\
$$\dot{q}_1=\frac{sq_1}{4(s^2+4)}-\frac{(s-2)q_3}{4(s^2+4)}+\frac{1}{2}\alpha q_4$$
$$\dot{q}_2=-\frac{\alpha q_1}{s^2+4}-\frac{sq_2}{4(s^2+4)}+\frac{\alpha(s+2)q_3}{2(s^2+4)}+\frac{(s-2)q_4}{4(s^2+4)}$$
$$\dot{q}_3=-\frac{(s-2)q_1}{4s(s^2+4)}+\frac{\alpha q_2}{2s}-\frac{(s^2+8)q_3}{4s(s^2+4)}$$
$$\dot{q}_4=\frac{\alpha (s+2)q_1}{2s(s^2+4)}+\frac{(s-2)q_2}{4s(s^2+4)}-\frac{\alpha q_3}{s^2+4}-\frac{(3s^2+8)q_4}{4s(s^2+4)}$$
The Galois group is virtually diagonal, and after gauge reduction we diagonalize it and we find the solutions
\begin{small}
$$(2+2\sqrt{s}+s)^{\frac{1}{4}}\left(1+\sqrt{s}+\sqrt{2+2\sqrt{s}+s}\right)^{\alpha}, (2-2\sqrt{s}+s)^{\frac{1}{4}}\left(1-\sqrt{s}+\sqrt{2-2\sqrt{s}+s}\right)^{\alpha},$$
$$(2+2\sqrt{s}+s)^{\frac{1}{4}}\left(1+\sqrt{s}-\sqrt{2+2\sqrt{s}+s}\right)^{\alpha},(2-2\sqrt{s}+s)^{\frac{1}{4}}\left(1-\sqrt{s}-\sqrt{2-2\sqrt{s}+s}\right)^{\alpha}$$
\end{small}
So $\mathbb{C}(\Sigma)=\mathbb{C}(s,\sqrt{2+2\sqrt{s}+s})$. The Galois group is of dimension $2$. The first integrals and vector fields produced by Theorem \ref{thmmain2} are
$$F_1=\frac{\sqrt{2+2\sqrt{s}+s}}{s^2+4}(s^3q_4^2+s^2q_2^2-s^2q_3^2-sq_1^2+4sq_1q_3-2sq_3^2+4sq_4^2-2q_1^2+4q_2^2$$
$$+2\sqrt{s}(s^2q_2q_4-sq_1q_3+sq_3^2+q_1^2-2q_1q_3+4q_2q_4))$$
$$F_2=\frac{\sqrt{2-2\sqrt{s}+s}}{s^2+4}(s^3q_4^2+s^2q_2^2-s^2q_3^2-sq_1^2+4sq_1q_3-2sq_3^2+4sq_4^2-2q_1^2+4q_2^2$$
$$-2\sqrt{s}(s^2q_2q_4-sq_1q_3+sq_3^2+q_1^2-2q_1q_3+4q_2q_4))$$
$$Y_1=(q_2+\sqrt{s}q_4)\sqrt{2+2\sqrt{s}+s}\frac{\partial}{\partial q_1}+\frac{q_1+q_3\sqrt{s}}{\sqrt{2+2\sqrt{s}+s}}\frac{\partial}{\partial q_2}+$$
$$\frac{(\sqrt{s}q_4+q_2)\sqrt{2+2\sqrt{s}+s}}{\sqrt{s}}\frac{\partial}{\partial q_3}+\frac{q_3\sqrt{s}+q_1}{\sqrt{s}\sqrt{2+2\sqrt{s}+s}}\frac{\partial}{\partial q_4} $$
$$Y_2=(q_2-\sqrt{s}q_4)\sqrt{2-2\sqrt{s}+s}\frac{\partial}{\partial q_1}+\frac{q_1-q_3\sqrt{s}}{\sqrt{2-2\sqrt{s}+s}}\frac{\partial}{\partial q_2}+$$
$$\frac{(\sqrt{s}q_4-q_2)\sqrt{2-2\sqrt{s}+s}}{\sqrt{s}}\frac{\partial}{\partial q_3}+\frac{q_3\sqrt{s}-q_1}{\sqrt{s}\sqrt{2-2\sqrt{s}+s}}\frac{\partial}{\partial q_4} $$
The coefficients of $F_1,F_2$ define the field $\mathbb{C}(\Sigma)$, and thus
$$\mathcal{J}_\Sigma/\mathcal{J}_{\bar{\Gamma}}\simeq\mathbb{C}\left(s,\sqrt{2+2\sqrt{s}+s}\right)$$
So this system is integrable on a neighbourhood of $q=0$. The meromorphic first integrals and vector fields are
$$F_1^2+F_2^2,\; (F_1^2-F_2^2)^2,\; F_1Y_1+F_2Y_2,\; (F_1^2-F_2^2)(F_1Y_1-F_2Y_2) $$

Remark that all our examples were linear in the $q$'s. This is not only because it is easier, but also because under the Galoisian hypothesis we make, the vector fields are linearisable using Theorem \ref{thmmain3}. Now the construction of such example relies on finding a polynomial $P\in\mathbb{C}(s)[X]$ with sufficiently many linear relations on the roots, i.e. the Galois group of $P$ can be recovered by its action on these relations. It appears that quite complicated groups are possible \cite{111}. This group in particular encodes how the stable/unstable reconnect in the neighbourhood of $\Gamma$.

\subsection{Linearisation at singular points}

The problem of extending our results to points of $\bar{\Gamma}\setminus \Gamma$ necessitates to prove the convergence of the series of Proposition \ref{propformal1} at singular points. In particular, these series at the limit at a singular points can give formal transformations.

\begin{prop}\label{propextend}
Let $X$ be a meromorphic vector field of a neighbourhood of $\bar{\Gamma}$, with $\Gamma$ an algebraic solution of $X$. Assume $NVE_1$ is Fuchsian, $\hbox{Gal}^0(NVE_k) \simeq \mathbb{C}^{l-1}$ for all $k\in \mathbb{N}^*$, $Mon^0(NVE_1)$ is Diophantine. Theorem \ref{thmmain2} applies and gives $l$ meromorphic commutative vector fields $Y$, $n-l$ independent meromorphic first integrals $F$ and a Riemann surface $\Sigma$ above $\bar{\Gamma}$. Let us consider $s_0\in\Sigma$ projecting on an equilibrium point of $X$ and assume the local monodromy group around $s_0$ is non resonant Diophantine. Then $Y,F$ extend meromorphically on a neighbourhood of $(s_0,0)$.
\end{prop}

\begin{proof}
Let us first make the gauge reduction of the vector field $X$. We can assume $s_0$ projects on $0\in\bar{\Gamma}$. After the gauge reduction, the system becomes
\begin{equation}\label{eqXred}
\dot{q}_i=\frac{A}{s} q_i + \frac{1}{s}R(s,q)
\end{equation}
with $A=\hbox{diag}(\alpha_1,\dots,\alpha_{n-1})$, $R(s,q)$ holomorphic in a neighbourhood of $q=0,s\in \Sigma$ with valuation in $q$ at least $2$. The pole at $0$ is always of order at most $1$ because $s=0$ is singular regular.

We consider the series $(\varphi_j(s,c))_{j=1\dots n-1}$ given by Proposition \ref{propformal1}
$$\varphi_j(s,c_1,\dots,c_{n-1})=\sum\limits_{i\in\mathbb{N}^{n-1}}  a_{j,i_1,\dots,i_{n-1}}(s) (c_1H_1(s))^{i_1}\dots (c_{n-1}H_{n-1}(s))^{i_{n-1}}$$
for $j=1\dots n-1$.

Let us first prove that $s=0$ is not a pole of $a_{j,i_1,\dots,i_{n-1}}(s)$. Recall that in proof of Proposition \ref{propformal1}, the $a_{j,i_1,\dots,i_{n-1}}(s)$ are computed using the formula
$$\frac{H_j(s)}{(c_1H_1(s))^{i_1}\dots (c_{n-1}H_{n-1}(s))^{i_{n-1}}}\int g_{j,i}(s)\frac{(c_1H_1(s))^{i_1}\dots (c_{n-1}H_{n-1}(s))^{i_{n-1}}}{H_j(s)} ds$$
where $g_{j,i}(s)$ is a coefficient of the series expansion of $X$ in $q$. And thus $g_{j,i}(s)$ has a pole of order at most $1$. We have moreover $H_i(s) \sim s^{\alpha_i}$ near $0$ (after possibly a scaling of the $H_i$) and the $\alpha_1,\dots,\alpha_{n-1}$ are non resonant. Now we know by hypothesis that the integral is of the form
$$a_{j,i}(s)\frac{(c_1H_1(s))^{i_1}\dots (c_{n-1}H_{n-1}(s))^{i_{n-1}}}{H_j(s)} \quad a_{j,i}(s)\in\mathbb{C}(\Sigma)$$

Let us now prove that $a_{j,i}$ has a limit at $s_0$. Making a series expansion near $0$ gives
$$\int g_{j,i}(s)\frac{(c_1H_1(s))^{i_1}\dots (c_{n-1}H_{n-1}(s))^{i_{n-1}}}{H_j(s)} ds =C+ bs^{i.\alpha-\alpha_j}+o(s^{i.\alpha-\alpha_j})$$
where $b$ is a constant depending on $g_{j,i}$ and $C$ the integration constant.
If $i.\alpha-\alpha_j$ is negative, then the valuation of $a_{j,i}$ at $s_0$ is non negative, and so $a_{j,i}$ converges at $s_0$. 
If $i.\alpha-\alpha_j$ is positive, then either $C=0$ and then $a_{j,i}$ converges at $s_0$. Or $C\neq 0$ and then $a_{j,i}\sim bs^{-i.\alpha+\alpha_j}$.

The local monodromy group near $s_0$ in $Mon^0(NVE_1)$ is generated by the diagonal matrix $\hbox{diag}(e^{2im\pi \alpha_1},\dots,e^{2im\pi \alpha_{n-1}})$ where $m\in\mathbb{N}^*$ is the ramification index of $\Sigma$ at $s_0$. The non resonance condition writes
$$i.m\alpha-m\alpha_j\notin \mathbb{Z} \quad \forall i\in\mathbb{N}^{n-1} \mid i\mid \geq 2$$
As $a_{j,i}\in\mathbb{C}(\Sigma)$, it admits a Puiseux series at $0$ in $\mathbb{C}[[s^{1/m}]]$. And thus $-m i.\alpha+m\alpha_j\in\mathbb{Z}$, which is impossible due to the non resonance condition.

We can now make a series expansion of the $a_{j,i}$ at $s_0$ (which projects to $0\in\bar{\Gamma}$), giving
$$\varphi_j(s,c_1,\dots,c_{n-1})=\!\!\sum\limits_{i\in\mathbb{N}^{n-1}}  \sum\limits_{i_n\in\mathbb{N}} a_{j,i_1,\dots,i_{n-1},i_n} s^{i_n/m} (c_1H_1(s))^{i_1}\dots (c_{n-1}H_{n-1}(s))^{i_{n-1}}$$
for $j=1\dots n-1$ and inverting the relation we obtain
$$c_iH_i(s)=\Phi_i(s,q)\quad \Phi_i\in\mathbb{C}[[s^{1/m},q]]$$
and then
$$c_is^{\alpha_i}=\frac{s^{\alpha_i}}{H_i(s)}\Phi_i(s,q)\in\mathbb{C}[[s^{1/m},q]]$$
So the application
$$(s,q) \mapsto \left(\frac{s^{\alpha_i}}{H_i(s)}\Phi_i(s,q)\right)_{i=1\dots n-1}$$
linearises the vector field \eqref{eqXred} near $(s,q)=(0,0)$.

Let us note $\tilde{s}=s^{1/m}$ and make a variable change in \eqref{eqXred}. The non resonance condition writes
$$i.m\alpha-m\alpha_j\notin \mathbb{Z} \quad \forall i\in\mathbb{N}^{n-1} \mid i\mid \geq 2$$
and with the new time $\tilde{s}$, $m\alpha$ are the eigenvalues of the leading matrix of the system. Now changing time by multiplying all the equations by $\tilde{s}$, we obtain a vector field with $0$ as equilibrium point, and $m\alpha_1,\dots,m\alpha_{n-1},1$ as eigenvalues. These eigenvalues satisfy the non resonance and Diophantine condition, and thus system is holomorphically linearisable near $(\tilde{s},q)=(0,0)$ and the coordinates change is unique \cite{112}. Thus it has to be
$$(\tilde{s},q) \mapsto \left(\frac{\tilde{s}^{m\alpha_i}}{H_i(\tilde{s}^m)}\Phi_i(\tilde{s}^m,q)\right)_{i=1\dots n-1}$$
which thus is holomorphic in $\tilde{s},q$. We then deduce that $\Phi$ is holomorphic in $\tilde{s},q$ near $(\tilde{s},q)=(0,0)$ and thus that the vector fields and first integrals $Y,F$ of $X$ can be extended to a neighbourhood of $s_0$.
\end{proof}

Remark that the resonance condition we put is slightly stronger than the minimal one possible, which is
$$i.\alpha-\alpha_j\neq 0 \quad \forall i\in\mathbb{N}^n,\; \mid i\mid \geq 2,\; \alpha_n=1$$
This is however necessary to ensure that $a_{j,i}$ remains regular at $0$. Indeed, if only the above condition is satisfied, then it is always possible to choose an $a_{j,i}$ regular at $0$. However if 
$i.\alpha-\alpha_j\in \mathbb{Z}$ this choice of integration constant could not lead to the solution in $\mathbb{C}(\Sigma)$, i.e. the local behaviour of $a_{j,i}$ is in $\mathbb{C}[[s^{1/m}]]$ but not algebraic.\\

\noindent
\textbf{Example}\\
Consider a sequence of polynomials $P_k$ of degree at most an affine function in $k$ and $\alpha\notin \mathbb{Q}$ such that
$$\int P_k(s) s^{3k/2} (1-s)^{\alpha k} ds \in \mathbb{C}\left(s^{1/2},(1-s)^{\alpha}\right) $$
This is typically the integral encountered in Proposition \ref{propformal1}, and the above condition ensure that the integral belongs to the same field. Moreover, there is an unique choice for integration constant such that
$$\frac{1}{s^{3k/2} (1-s)^{\alpha k}}\int P_k(s) s^{3k/2} (1-s)^{\alpha k} ds \in \mathbb{C}\left(s^{1/2}\right)$$
However, in Proposition \ref{propextend} (whose non resonance hypothesis is not satisfied), we also need this expression to be regular at $0$. The condition writes down
$$\int\limits_0^1 P_k(s) s^{3k/2} (1-s)^{\alpha k} ds=0$$
and nothing ensures it is satisfied for all $k$. If not, we can expect the valuation of the $a_{j,i}$ to go to $-\infty$, meaning that $s=0$ would be an essential singularity for our vector fields/first integrals, even if the spectrum $1,3/2$ allows linearisation near $s=0$.\\

Remark also that the extension depends a priori on the sheave of $\Sigma$ on which $s_0$ is and not only its projection. This is due to the fact that the non resonance condition is not the same if the ramification index $m$ is not the same, as in the following example
$$\alpha_1=\sqrt{2},\; \alpha_2=2\sqrt{2}+1/2$$
We have $k.\alpha-\alpha_j\notin \mathbb{Z}$ but $4\alpha_1-2\alpha_2\in\mathbb{Z}$ and thus is resonant with ramification index $m=2$. So if we now want to combine this result with Proposition \ref{propcover}, we need to be able to extend the vector fields and first integrals of Proposition \ref{propformal1} for all sheaves. So to extend the vector fields and first integrals of Proposition \ref{propcover} near an equilibrium point $s_0\in\bar{\Gamma}$, the local monodromy group around any point of $\pi^{-1}(s_0)$ should be non resonant Diophantine.

\section{Conclusion}

We proved an inverse of Ayoul Zung Theorem under a small divisor condition and a non resonance condition or strong additional Galoisian condition. This strong Galoisian condition $\hbox{Gal}^0(NVE_k) \simeq \mathbb{C}^{l-1}$ is in fact implied by the non resonance condition, suggesting it should not be so rare after all. It happens that the same approach allowed us to prove a linearisation Theorem \ref{thmmain3}. This Theorem is closely related to Theorems \ref{thmmain1},\ref{thmmain2} as after linearisation, construction of commuting vector field and first integrals is possible. This suggests that similar more general inverses of Ayoul Zung Theorem are related to normal forms of holomorphic vector fields depending on time on a neighbourhood of $0$. The example
$$\dot{q}_1=\frac{\alpha}{t}q_1+\frac{1}{s}q_1^2q_2\quad \dot{q}_2=-\frac{\alpha}{t}q_2-\frac{1}{s}q_1q_2^2$$
which is not linearisable but integrable suggests more general series than in Proposition \ref{propformal1} should be considered, in particular allowing formal first integrals in the exponents. This would produce resonant normal forms for the maps of Ziglin group $\hbox{Zig}^0(X)$
$$q\rightarrow \phi(F_1(q),\dots,F_p(q))q$$
where $F_1,\dots,F_p$ are invariants of the map and $\phi$ formal series. This resonant normal form does not remove all the resonant monomials, and so we would need a definition of canonic coordinates change for the uniqueness and a theorem for the convergence of such formal coordinates change. We can hope that such approach would be sufficient to invert the Ayoul Zung Theorem without requiring additional Galoisian conditions other than virtual Abelianity of the $VE_k$.

The completion of the vector fields and first integrals $Y,F$ of Theorems \ref{thmmain1},\ref{thmmain2} near singular points of $X$ remains elusive and probably generically not possible. Indeed, at a singular point of $X$, nothing seems to forbid the valuation in $s$ of the coefficients of the series expansion of $X$ to go to $-\infty$, giving an essential singularity to our coordinate change map $\Phi$.

\bibliography{bibthese}

%\begin{thebibliography}{9}
%  \bibitem{Lange:2007}
%  T.~Lange and I.~E. Shparlinski,
%  Distribution of some sequences of points on elliptic curves,
%  \emph{J. Math. Cryptol.} \textbf1 \left(2007\right), 1--11.
%\end{thebibliography}

\end{document}